\pdfoutput=1

\documentclass[reqno,a4paper,final]{amsart}

\usepackage[utf8]{inputenc}
\usepackage[T1]{fontenc}
\usepackage[english]{babel}
\usepackage{ifthen}
\usepackage{enumitem}
\usepackage{dsfont}
\usepackage{mathtools}
\usepackage{amssymb}
\usepackage[babel]{csquotes}
\usepackage[backend=biber,style=alphabetic]{biblatex}
\usepackage[protrusion=true,expansion=true,babel=true,final]{microtype}
\usepackage[unicode,bookmarksopen]{hyperref}

\numberwithin{equation}{section} 

\newtheorem{lemma}{Lemma}[section]

\newtheorem{proposition}[lemma]{Proposition}
\newtheorem{theorem}[lemma]{Theorem}

\theoremstyle{definition}
\newtheorem{definition}[lemma]{Definition}
\newtheorem{example}[lemma]{Example}
\newtheorem{remark}[lemma]{Remark}

\newtheorem{convention}[lemma]{Convention}

\newlist{thm_enum}{enumerate}{1}
\setlist[thm_enum]{label=\normalfont(\alph*)}
\newlist{def_enum}{enumerate}{1}
\setlist[def_enum]{label=\normalfont(\roman*)}
\newlist{equiv_enum}{enumerate}{1}
\setlist[equiv_enum]{label=\normalfont(\roman*)}

\newcommand{\IZ}{\mathbb{Z}}

\newcommand{\IN}{\mathbb{N}}
\newcommand{\IR}{\mathbb{R}}
\newcommand{\IC}{\mathbb{C}}
\newcommand{\IE}{\mathbb{E}}

\newcommand{\IP}{\mathbb{P}}

\newcommand{\abs}[1]{\left\lvert#1\right\rvert}
\newcommand{\normalabs}[1]{\lvert#1\rvert}

\newcommand{\norm}[1]{\left\lVert#1\right\rVert}
\newcommand{\normalnorm}[1]{\lVert#1\rVert}
\newcommand{\biggnorm}[1]{\biggl\lVert#1\biggr\rVert}

\newcommand{\R}[2][\empty]{
	\ifthenelse{\equal{#1}{\empty}}
		{\mathcal{R}\left\{#2\right\}}
		{\mathcal{R}_{#1}\left\{#2\right\}}
}

\makeatletter
\newcommand{\LeftEqNo}{\let\veqno\@@leqno}
\makeatother

\renewcommand{\d}{\mathop{}\!d}
\renewcommand{\Re}{\operatorname{Re}}
\renewcommand{\Im}{\operatorname{Im}}
\renewcommand{\epsilon}{\varepsilon}

\let\temp\phi
\let\phi\varphi
\let\varphi\temp

\DeclareMathOperator{\E}{\IE}

\DeclareMathOperator{\Id}{Id}

\DeclareMathOperator{\divv}{div}
\DeclareMathOperator{\Tr}{Tr}

\allowdisplaybreaks[4]

\DeclareSourcemap{
  \maps[datatype=bibtex, overwrite]{
    \map{
      \step[fieldset=abstract, null]
      \step[fieldsource=entrykey,match=Mer99,final] %
      \step[fieldset=shorthand,fieldvalue=LeM99] %
    }
  }
}

\renewbibmacro{publisher+location+date}{%
	\printlist{publisher}
	\setunit*{\addcomma\space}
	\printlist{location}
  	\setunit*{\addcomma\space}
  	\usebibmacro{date}
\newunit} 

\newbibmacro{string+doi+url}[1]{%
	\iffieldundef{doi}{%
			\iffieldundef{url}{#1}{\href{\thefield{url}}{#1}}%
		}%
		{\href{http://dx.doi.org/\thefield{doi}}{#1}}
}%

\renewbibmacro{in:}{}

\DeclareFieldFormat*{title}{\usebibmacro{string+doi+url}{\mkbibemph{#1}}}
\DeclareFieldFormat*{booktitle}{#1}
\DeclareFieldFormat[article]{volume}{\textbf{#1}\addcomma}
\DeclareFieldFormat[article]{number}{\addspace no.~#1}
\DeclareFieldFormat[article]{pages}{#1}
\DeclareFieldFormat[book,incollection]{volume}{}
\DeclareFieldFormat[book,incollection]{series}{#1\addcomma\space vol.~\thefield{volume}}
\DeclareFieldFormat{journaltitle}{#1} 
\DeclareFieldFormat{url}{}
\DeclareFieldFormat{eprint}{arXiv: \href{http://arxiv.org/abs/#1}{#1}} 

\renewcommand*{\bibfont}{\footnotesize}

\begin{document}

\title[Maximal Regularity for Rough Elliptic Operators]{Non-Autonomous Maximal $L^p$-Regularity for Rough Divergence Form Elliptic Operators}

\begin{abstract}
	We obtain $L^p(L^q)$ maximal regularity estimates for time dependent second order elliptic operators in divergence form with rough dependencies in the spatial variables.
\end{abstract}

\author{Stephan Fackler}
\address{Institute of Applied Analysis, University of Ulm, Helmholtzstr.\ 18, 89069 Ulm}
\email{stephan.fackler@uni-ulm.de}
\thanks{This work was supported by the DFG grant AR 134/4-1 ``Regularität evolutionärer Probleme mittels Harmonischer Analyse und Operatortheorie''. The author thanks Manuel Bernhard and Manfred Sauter for various fruitful discussions on elliptic operators.}
\keywords{non-autonomous maximal regularity, parabolic equations in divergence form, Banach scales, extrapolation spaces}
\subjclass[2010]{Primary 35B65; Secondary 35K10, 35B45, 47D06.}

\maketitle

\section{Introduction}
	
	We treat the maximal $L^p$-regularity of the non-autonomous Cauchy problem
	\begin{equation*}
		\LeftEqNo
		\label{nacp}\tag{NACP}
		\left\{
		\begin{aligned}
			\dot{u}(t) + A(t)u(t) & = f(t) \\
			u(0) & = u_0,
		\end{aligned}
		\right.
	\end{equation*}
	where $(A(t))_{t \in [0,T]}$ for some $T \in (0, \infty)$ are elliptic operators in divergence form on $L^q$ for some $q \in (1, \infty)$. One says that the problem~\eqref{nacp} has maximal $L^p$-regularity for $p \in (1, \infty)$ and for initial values in some subspace $Z \hookrightarrow X$ if for all right hand sides $f \in L^p([0,T];L^q)$ and all initial values $u_0 \in Z$ there exists a unique solution $u$ in its maximal regularity space, i.e.\ $u \in L^p([0,T];L^q)$ with $u(t) \in D(A(t))$ for almost all $t \in [0,T]$ and both the distributional derivative $\dot{u}$ and $A(\cdot)u(\cdot)$ lie in $L^p([0,T];L^q)$.
	
	Maximal regularity is very useful for the study of quasilinear partial differential equations as it allows the application of powerful linearization techniques (see for example~\cite{Pru02} or~\cite{Ama05}). In the autonomous case $A(t) = A$, the theory of maximal $L^p$-regularity is far developed.
	
	In the non-autonomous case, there are two fundamentally different situations, namely those of time dependent and time independent domains $D(A(t))$. In the time independent case very convenient criteria can be deduced with perturbation techniques if the operators depend continuously on the time variable (\cite{PruSch01} and~\cite{ACFP07}). However, one can go further: for second order elliptic operators in non-divergence form maximal $L^p$-regularity does even hold for time measurable coefficients if the spatial components lie in $VMO$. There have been many results in this direction in the past years, e.g.\ \cite{Kry08}, \cite{KimKry07}, \cite{Kim07}. Very recently, Gallarati and Veraar proved an abstract $L^p$-maximal regularity criterion for time independent domains that essentially covers these cases in~\cite{GalVer14} and~\cite{GalVer15}.
	
	However, in the setting of time dependent domains, the theory is far less developed. The best available criteria use the so-called Acquistapace--Terreni conditions, see \cite{HieMon00b} in the Hilbert space and \cite{PorStr06} in the Banach space case. Applications are rarely found as the condition looks non-intuitive and therefore difficult to verify. In contrast to this first impression, in the Hilbert space case, very convenient sufficient criteria for these conditions in terms of form methods were found (see~\cite{OuhSpi10} and~\cite{HaaOuh15}). In this work we give a similar criterion in general UMD spaces in terms of Banach scales.
	
	To illustrate its power we apply it to two concrete non-autonomous model problems. First as a model case for more general boundary problems, we obtain for $p,q \in (1, \infty)$ complete $L^q(L^p)$ maximal regularity estimates for the Laplacian with non-autonomous Robin boundary conditions assuming Hölder continuity in the time and Lipschitz continuity in the space variable. In the second model case and the heart of the article we show for $p \in [2, \infty)$ and $q \in (1, \infty)$ maximal $L^q(L^p)$ regularity for time dependent pure second order elliptic operators in divergence form subject to Dirichlet boundary conditions if the spatial coefficients -- analogous to the known results for time independent domains -- lie in $VMO$, whereas the time regularity is assumed to be $\alpha$-Hölder for $\alpha > \frac{1}{2}$. To the best knowledge of the author no results in this direction have been known in the non-Hilbert space case. Further, the results seem to be rather optimal: whereas the $VMO$ assumption in the spatial variables is a common theme for divergence and non-divergence form operators, the Hölder assumption in time is optimal for general operators by the recent negative solution of Lions' problem by the author~\cite{Fac16c}. We remark that in~\cite[Section~4]{HieMon00} the same result is obtained under the stronger assumption that the spatial coefficients are $C^1$. Further, for a weaker notion of regularity, a priori estimates were obtained in~\cite{Kry07} and~\cite{Don10} only assuming sufficiently small bounded mean oscillation in the spatial variables.
	
	Let us shortly sketch the structure of our presentation: in the next section we give a detailed account of the needed mathematical concepts: interpolation functors, imaginary powers, Banach scales, $\mathcal{R}$-boundedness and the Acquistapace--Terreni framework. Afterwards we present a general maximal regularity result in Theorem~\ref{thm:abstract_mr_non-zero}. In the main part we deal with the verification of the assumptions of Theorem~\ref{thm:abstract_mr_non-zero} for elliptic operators in divergence form and then prove the claimed concrete maximal regularity results.
	
\section{Mathematical Background}

\subsection{Interpolation functors and fractional domain spaces}

	From now on all Banach spaces $X$ and $Y$ are assumed to be complex. The Banach space of all bounded linear operators between $X$ and $Y$ is denoted by $\mathcal{B}(X)$. For a closed operator $(A,D(A))$ on $X$ we write $\sigma(A)$ for the spectrum of $A$ and $\rho(A)$ for the resolvent set of $A$. For $\lambda \in \rho(A)$ the resolvent of $A$ is defined as $R(\lambda, A) = (\lambda - A)^{-1}$. Moreover, for $\varphi \in (0, \pi)$ let $\Sigma_{\varphi} \coloneqq \{ z \in \IC \setminus \{0 \}: \abs{\arg z} < \varphi \}$ be the sector of opening angle $\varphi$.
	
	\begin{definition}
		A closed densely defined operator $(A,D(A))$ on a Banach space $X$ is called \emph{sectorial} if $A$ has dense range and if there exists $\varphi \in (0, \frac{\pi}{2})$ with $\sigma(A) \subset \overline{\Sigma_{\varphi}}$ and
			\[
				\sup_{\lambda \not\in \overline{\Sigma_{\varphi}}} \norm{\lambda R(\lambda, A)} < \infty.
			\]
	\end{definition}
	
	Recall that an operator $A$ is sectorial if and only if $-A$ generates a bounded analytic $C_0$-semigroup on $X$. 
	For a sectorial operator $A$ and $z \in \IC$ one can define the fractional powers $A^z$. For details on sectorial operators and their fractional powers we refer to~\cite{CarSan01} and~\cite{Haa06}.
	
	\begin{definition}
		A sectorial operator $(A,D(A))$ on a Banach space $X$ has \emph{bounded imaginary powers} if $A^{it}$ is a bounded operator for all $t \in \IR$.
	\end{definition}

	If $A$ has bounded imaginary powers, then it can be shown that $(A^{it})_{t \in \IR}$ is a strongly continuous group. Hence, there exist $M \ge 0$ and $\omega \ge 0$ with
		\[
			\normalnorm{A^{it}}_{\mathcal{B}(X)} \le M e^{\omega \abs{t}} \qquad \text{for all } t \in \IR.
		\]
		
	Let us now turn to interpolation functors. Two Banach spaces $X$ and $Y$ are called a \emph{Banach couple} if there exist a Hausdorff topological vector space in which both $X$ and $Y$ embed. We write $(X,Y)$ for such a couple. For a Banach couple $(X,Y)$ one can define its sum $X + Y$, which is a Banach space for the norm $\norm{z}_{X+Y} \coloneqq \inf \{ \norm{x}_X + \norm{y}_Y: z = x + y, x \in X, y \in Y \}$. Further, the intersection $X \cap Y$ is a Banach space for the norm $\norm{z}_{X \cap Y} = \norm{x}_X + \norm{y}_Y$. The class of Banach couples becomes a category if we let the morphisms $T\colon (X_1,Y_1) \to (X_2, Y_2)$ be those $T \in \mathcal{B}(X_1 + Y_1, X_2 + Y_2)$ that satisfy $TX_i \subset Y_i$ for $i = 1, 2$.
	
	\begin{definition}
		An \emph{interpolation functor} is a functor $\mathcal{F}$ from the category of Banach couples into the category of Banach spaces that satisfies
			\begin{def_enum}
				\item $X \cap Y \subset \mathcal{F}((X,Y)) \subset X + Y$ for all Banach couples $(X,Y)$.
				\item $\mathcal{F}(T) = T_{|\mathcal{F}((X_1,Y_1))}$ for all morphisms $T\colon (X_1, Y_1) \to (X_2, Y_2)$ of interpolation couples.
			\end{def_enum}
		An interpolation functor $\mathcal{F}$ has \emph{order} $\theta \in [0,1]$ if there exists a constant $c \ge 0$ such that for all morphisms $T\colon (X_1, Y_1) \to (X_2, Y_2)$ of Banach couples one has
			\[
				\norm{T}_{\mathcal{B}(\mathcal{F}((X_1,Y_1)),\mathcal{F}((X_2, Y_2)))} \le c \norm{T}_{\mathcal{B}(X_1,X_2)}^{1-\theta} \norm{T}_{\mathcal{B}(Y_1,Y_2)}^{\theta}.
			\]
		Further, $\mathcal{F}$ is exact of order $\theta \in [0,1]$ if one can choose $c = 1$. The functor $\mathcal{F}$ is called \emph{regular} if $X \cap Y$ is dense in $\mathcal{F}((X,Y))$ for all interpolation couples $(X,Y)$.
	\end{definition}
	
	The most important examples of exact interpolation functors of order $\theta$ are the complex interpolation functors $[\cdot,\cdot]_{\theta}$ and the real interpolation functors $(\cdot,\cdot)_{\theta, q}$ for $q \in [1, \infty]$. The complex interpolation functors are regular for $\theta \in (0,1)$, whereas the real interpolation functors are regular provided $\theta \in (0,1)$ and $q \in [1, \infty)$. We now present the definition of the complex interpolation method in view of the next result. For more details and the definition of the real interpolation method we refer to the monographs~\cite{BerLoe76} and~\cite{Tri78}.
	
	We set $S \coloneqq \{ z \in \IC: 0 < \Re z < 1 \}$. For a Banach couple $(X,Y)$ we denote by $\mathcal{J}(X,Y)$ the space of all continuous and bounded functions $f\colon \overline{S} \to X + Y$ whose restrictions to $S$ are analytic and for which $f(i\cdot)\colon \IR \to X$ and $f(1+i\cdot)\colon \IR \to Y$ are bounded continuous functions. The space $\mathcal{J}(X,Y)$ is a Banach space for the norm $\norm{f}_{\mathcal{J}(X,Y)} = \max \{ \sup_{t \in \IR} \norm{f(it)}_X, \sup_{t \in \IR} \norm{f(1+ it)}_Y \}$.  
	
	\begin{definition}
		For a Banach couple $(X,Y)$ and $\theta \in [0,1]$ we let
			\[
				[X,Y]_{\theta} \coloneqq \{ f(\theta): f \in \mathcal{J}(X,Y) \}
			\]
		endowed with the norm
			\[
				\norm{x}_{\theta} \coloneqq \inf \{ \norm{f}_{\mathcal{J}(X,Y)}: f \in \mathcal{J}(X,Y), f(\theta) = x \}.
			\]
	\end{definition}
	
	For a sectorial operator with bounded imaginary powers one can identify its fractional domains with complex interpolation spaces up to isomorphisms. In the next result we give the following quantitative version: if the bounded imaginary powers are uniformly bounded, then the identification holds with a uniform constant.
	
	We make the following agreement: for a sectorial operator $A$ on some Banach space $X$ we use the norm $\norm{x}_{D(A)} \coloneqq \norm{Ax}_X$.
	
	\begin{proposition}\label{prop:bip_fractional}
		Let $(A_j, D(A_j))_{j \in J}$ be a family of invertible sectorial operators on some Banach space $X$. Suppose further that the operators $A_j$ have bounded imaginary powers with
			\[
				\normalnorm{A_j^{it}} \le M e^{\omega \abs{t}} \qquad \text{for all }  t \in \IR \text{ and } j \in J
			\]
		for some constants $\omega \ge 0$ and $M \ge 0$. Then for all $\theta \in (0,1)$ there exists a constant $C > 0$ with
			\[
				C^{-1} \norm{x}_{D(A_j^{\theta})} \le \norm{x}_{[X,D(A_j)]_{\theta}} \le C \norm{x}_{D(A_j^{\theta})} \quad \text{for all } x \in D(A_j^{\theta}) \text{ and } j \in J.
			\]
	\end{proposition}
	\begin{proof}
		This is a consequence of the classical proof. For the sketchy proof given here we follow the lines and notation in~\cite[Theorem~6.6.9]{Haa06}. For given $x \in D(A_j^{\theta})$ the functions $f_j(z) = e^{(z-\theta)^2}  A_j^{-z} A_j^{\theta}x$ lie in $\mathcal{J}(X, D(A_j))$ and satisfy
		\begin{align*}
			\norm{f_j(it)}_X & \le M e^{\theta^2} e^{-t^2} e^{\omega \abs{t}} \normalnorm{A_j^{\theta} x}_X \\
			\norm{f_j(1+it)}_{D(A_j)} & \le M e^{(1-\theta)^2} e^{-t^2} e^{\omega \abs{t}} \normalnorm{A_j^{\theta} x}_X.
		\end{align*}
		This shows the uniformity of the second estimate. The converse inequality is shown for all $x$ in the set $D(A_j)$ and then follows by density. In fact, for all $f_j$ with $f_j(\theta) = x$ in a set $\mathcal{V}_j \subset \mathcal{J}(X,D(A_j))$ that satisfies $\norm{x}_{[X,D(A_j)]_{\theta}} = \inf_{f_j \in \mathcal{V}_j} \norm{f_j}_{\mathcal{J}(X,D(A_j))}$ one shows that
		\begin{align*}
			\MoveEqLeft \norm{x}_{D(A_j^{\theta})} \le M \sup_{s \in \IR} \max \{ e^{\theta^2} e^{-s^2} e^{\omega \abs{s}} \norm{f_j(is)}_X, e^{(1-\theta)^2} e^{-s^2} e^{\omega \abs{s}} \norm{f_j(1 + is)}_{D(A)} \}.
		\end{align*}
		This shows that for some constant $C \ge 0$ independent of $j \in J$ and $f_j \in \mathcal{V}_j$ one has $\norm{x}_{D(A_j^{\theta})} \le C \norm{f_j}_{\mathcal{J}(X, D(A_j))}$ and consequently $\norm{x}_{D(A_j^{\theta})} \le \norm{x}_{[X,D(A_j^{\theta})]}$.
	\end{proof}
	
\subsection{Second order elliptic problems in divergence form}

	In the main part we will consider realizations of differential operators of the form
		\begin{equation}
			\label{eq:elliptic_operator}
			Lu = -\divv(A\nabla u + a \cdot \nabla u) + b \cdot \nabla u + c_0 u
		\end{equation}
	with bounded measurable real coefficients $A = (a_{ij})$, $a = (a_i)$, $b = (b_i)$ and $c_0$ subject to Dirichlet, Neumann or Robin boundary conditions, i.e.\ $\mathcal{B}u = 0$ on $\partial \Omega$ for some sufficiently regular domain $\Omega \subset \IR^N$, where 
		\begin{equation*}
			\mathcal{B}u = \begin{cases}
				u & \text{for Dirichlet}, \\
				\sum_{i=1}^N (\sum_{j=1}^N a_{ij} \partial_j u) + a_i u \nu_i & \text{for Neumann}, \\
				\sum_{i=1}^N (\sum_{j=1}^N a_{ij} \partial_j u) + a_i u \nu_i + \beta u & \text{for Robin}.
			\end{cases}
		\end{equation*}

\subsubsection{The form method}\label{sec:form_method}
	
	We begin with introducing the form method. For further details and proofs we refer to~\cite{Ouh05}. Let $V$ be a real Hilbert space and let $a\colon V \times V \to \IR$ be a \emph{form}, i.e.\ a bilinear mapping $V \times V \to \IR$. Suppose further that $V$ is densely embedded in a second Hilbert space $H$. The form $a$ induces an unbounded operator on $H$ defined as
	\begin{align*}
		D(A) & = \{ u \in V: \exists f \in H: a(u,v) = (f | v)_{H} \text{ for all } v \in V \}, \\
		Au & = f.
	\end{align*}
	Suppose further that for all $u, v \in V$ one has
	\begin{align*}
		\abs{a(u,v)} &\le M \norm{u}_V \norm{v}_V, \qquad a(u,u) \ge \alpha_0 \norm{u}_V^2
	\end{align*}
	for some positive constants $\alpha_0$ and $M$. Then one can show that the operator $-A$ generates an analytic semigroup on the complexification of $H$. Further, the analyticity angle of the generated semigroup depends only on upper bounds for $M$ and $\alpha_0^{-1}$. If we replace $a$ by the form $b = a + \omega(\cdot | \cdot)_H$ for some $\omega \in \IR$, then the operator associated to $b$ is given by $A + \omega$. Hence, we may apply the above generation result to shifted forms if these satisfy the above estimates. %
	
	If $H = L^2(\Omega)$, we want to extrapolate the operator $A_2 = A$ given by the form $a$ to the spaces $L^p(\Omega)$. Of course, in general this is not possible. However, suppose that the semigroup $(T_2(t))_{t \ge 0}$ generated by $-A_2$ leaves $L^p(\Omega) \cap L^2(\Omega)$ invariant and satisfies $\norm{T_2(t)f}_p \le C \norm{f}_p$ for some $C \ge 0$ for all $t \in [0,1]$ and all $f \in L^p(\Omega) \cap L^2(\Omega)$. Then, provided $p \in [1, \infty)$, $(T_2(t))_{t \ge 0}$ induces a compatible $C_0$-semigroup $(T_p(t))_{t \ge 0}$ whose generator we denote by $-A_p$. Clearly, one has $A_p = A_2$ on the intersection of their domains. In practice, the invariance can be checked directly on the form. The most simple case is here the invariance of $L^2(\Omega) \cap L^{\infty}(\Omega)$. Moreover, positive results can be obtained by a more sophisticated approach that is able to treat the case $p \in (1, \infty)$ directly~\cite{Nit12}.

\subsubsection{The form method for elliptic operators}\label{sec:form_elliptic_operator}
	
	Let $\Omega$ be a domain in $\IR^N$. Coming back to the elliptic operator~\eqref{eq:elliptic_operator}, for Dirichlet and Neumann boundary conditions we introduce the form
		\begin{equation}
			\label{eq:form_dn}
			a_0(u,v) = \sum_{i=1}^N \int_{\Omega} \biggl( \sum_{j=1}^N a_{ij} \partial_j u + a_i u \biggr) \partial_i v + \int_{\Omega} \biggl( \sum_{i=1}^N b_i \partial_i u + c_0 u \biggr) v,
		\end{equation}
	where the form domain is $W_0^{1,2}(\Omega)$ in the Dirichlet and $W^{1,2}(\Omega)$ in the Neumann case. As always from now on, we assume that all the above coefficients bounded measurable real functions. Further, we assume that the ellipticity condition
		\begin{equation*}
			\label{eq:ellipticity_condition}
			\tag{E}
			\sum_{i,j=1}^N a_{ij}(x) \xi_i \xi_j \ge \alpha_0 \abs{\xi}^2 \qquad \text{for } \xi \in \IR^N
		\end{equation*}
	holds almost everywhere for some $\alpha_0 > 0$. In the case of Robin boundary conditions we use the form
		\begin{equation}
			\label{eq:form_robin}
			a_1(u,v) = a_0(u,v) + \int_{\partial \Omega} \beta u v \d\mathcal{H}_{N-1},
		\end{equation}
	where $\mathcal{H}_{N-1}$ is the $N-1$-dimensional Hausdorff measure on $\partial \Omega$ and $\beta$ is a bounded measurable function on $\partial \Omega$. We will always assume in the Robin case that $\Omega$ is a bounded Lipschitz domain. It is known that in this case $\mathcal{H}_{N-1}$ coincides with the surface measure on $\partial \Omega$~\cite[Proposition~12.9]{Tay06} and that $W^{1,2}(\Omega)$ is the appropriate form domain.
	
	In all these cases, by the general form theory, the form induces an operator $A_2$ and $-A_2$ generates an analytic semigroup. Moreover, one can show that for all $p \in (1, \infty)$ these semigroups leave $L^2(\Omega) \cap L^p(\Omega)$ invariant and therefore induce compatible generators $-A_p$ as described in the last subsection (see~\cite{Dan00} and~\cite{Ouh05}). 

\subsection{Interpolation and extrapolation scales}

	We now give an introduction to the concept of interpolation and extrapolation scales. These are closely related to the study of weak solutions of PDEs. We only present those parts of the theory necessary for our purposes. A complete development of the theory can be found in \cite[Chapter~V]{Ama95}.

\subsubsection{The power scale}
			
	Let $A$ be an invertible sectorial operator on some Banach space $X$. As a first step we define a scale of Banach spaces associated to $X$ and $A$.
	
	For $n \in \IN$ we endow $D(A^n)$ with the norm $\norm{x}_{D(A^n)} = \norm{A^n x}_X$. This norm is equivalent to the graph norm of $A$ and we obtain a Banach space denoted by $X_{n,A}$. Further for negative integer values, let $\norm{x}_{-n} = \norm{A^{-n} x}_X$. Then $(X, \norm{\cdot}_{-n})$ is a normed vector space. By choosing compatible completions $X_{-n,A}$ and denoting $X_{0,A} \coloneqq X$, we obtain for each $m \in \IN$ a family of Banach spaces $(X_{n,A})_{n \ge -m}$ with
		\[
			\cdots \xhookrightarrow{d} X_{2,A} \xhookrightarrow{d} X_{1,A} \xhookrightarrow{d} X \xhookrightarrow{d} X_{-1,A} \xhookrightarrow{d} \cdots \xhookrightarrow{d} X_{-m+1,A} \xhookrightarrow{d} X_{-m,A}.	
		\]
	Further, if $X$ is reflexive, which we assume from now on, one can extend this scale to the whole range of integers by using duality arguments~\cite[Theorem~V.1.4.9]{Ama95}. %
	For our further needs the finite version will be sufficient.

	As a second step we extend the action of $A$ to $X_{n,A}$. By definition, for $n \ge 0$, the map $x \mapsto Ax$ is an isometric isomorphism from $X_{n+1,A}$ to $X_{n,A}$. Similarly, for $n \le 0$, $A\colon D(A) \mapsto X_{-(n+1),A}$ extends to an isometric isomorphism from $X_{-n,A}$ onto $X_{-(n+1),A}$. All these extensions are compatible. To distinguish between the actions of $A$ on the different spaces, we denote for $n \in \IZ$ the action of $A$ as an element of $\mathcal{B}(X_{n+1,A}, X_{n,A})$ by $A_n$. Furthermore, we see $A_n$ as an unbounded operator on $X_{n,A}$ with domain $X_{n+1,A}$. Using these agreements, the original operator $A$ coincides with $A_0$. Moreover, one can show that for all $n \in \IZ$ one has $\rho(A_n) = \rho(A)$ and that $A_n$ satisfies the exact same spectral estimates as $A$ does~\cite[Lemma~V.13.7]{Ama95}. We call the family $(X_{n,A}, A_n)_{n \in \IZ}$ the \emph{power scale} associated to $A$.

	As a third step we study duality properties. Recall that for a densely defined operator $A$ on $X$ its adjoint is the unbounded operator on the dual $X'$ defined by
	\begin{align*}
		D(A') & = \{ x' \in X': \exists y' \in X': \langle x', Ax \rangle = \langle y', x \rangle \text{ for all } x \in D(A) \} \\
		A'x' & = y'.
	\end{align*}
	Suppose again that $A$ is an invertible sectorial operator on a reflexive Banach space $X$. Then its adjoint $A'$ is an invertible sectorial operator on $X'$ with $\rho(A') = \rho(A)$. Therefore one can define the scale $(X_{n,A'})_{n \in \IZ}$ associated to $A'$. For each $n \in \IZ$ this induces compatible pairings
	\[
		X_{n,A} \times X'_{-n,A'} \ni (x,x') \mapsto \langle A^n x, (A_{-n}')^{-n} x' \rangle_{X,X'}.
	\]
	Further, the spaces $X_{n,A}$ are reflexive and one has $(X_{n,A})' \simeq X'_{-n,A'}$ and $(A_n)' = A'_{-n}$ with respect to the above pairing~\cite[Theorem~V.1.4.9]{Ama95}.

\subsubsection{Interpolation-Extrapolation scales}

	Let $A$ be an invertible sectorial operator on a reflexive Banach space $X$. In this section we extend the power scale associated to $A$ to the real numbers via interpolation theory. For this notice that $(X_{n,A}, X_{n+1,A})_{n \in \IZ}$ is a Banach couple for $n \in \IZ$. Hence, we can apply interpolation functors to these couples. To simplify notation and our presentation, we make the following agreement.
	
	\begin{convention}
		From now on let $(\cdot, \cdot)_{\theta}$ either be the family of complex interpolation functors $[\cdot,\cdot]_{\theta}$ or or the family of real interpolation functors $(\cdot,\cdot)_{\theta,q}$ for some $q \in (1, \infty)$.
	\end{convention}
	
	\begin{remark}
		Note that some of the following results and methods can be generalized to general regular interpolation functors, varying methods in $\theta$ or even certain abstract Banach scales, like the fractional power scale. 
		For a more general treatment of interpolation-extrapolation scales we refer to~\cite[Section~V.1.5]{Ama95}.
	\end{remark}
	
	\begin{definition}
		For $\alpha \in \IR$ and $n \in \IZ$ with $n < \alpha < n + 1$ set
			\[
				X_{\alpha,A} = (X_{n,A}, X_{n+1,A})_{\alpha - n} \qquad \text{and} \qquad A_{\alpha} = \text{part of } A_{n} \text{ in } X_{\alpha,A}.
			\]
		The family $(X_{\alpha,A}, A_{\alpha})_{\alpha \in \IR}$ is called the \emph{interpolation-extrapolation scale} associated to $A$.
	\end{definition}
	
	Loosely spoken, the properties of the power scale pass to the interpolation-extrapolation scale. In fact, for $\alpha < \beta$ one obtains dense natural embeddings $X_{\beta,A} \xhookrightarrow{} X_{\alpha,A}$ compatible with the actions of the operators $A_{\alpha}$. Furthermore for all $\alpha \in \IR$ one has $\rho(A_{\alpha}) = \rho(A)$ and $A_{\alpha}$ satisfies the same sectorial estimates as $A$~\cite[Proposition~1.5.5]{Ama95}. 
	
	Since $A$ is sectorial, there exist $\varphi \in (0, \frac{\pi}{2})$ and $C \ge 0$ with
		\[
			\norm{R(\lambda,A)}_{\mathcal{B}(X,X)} \le C \abs{\lambda}^{-1} \qquad \text{and} \qquad \norm{R(\lambda,A)}_{\mathcal{B}(X,D(A))} \le C
		\]
	for all $\lambda \not\in \overline{\Sigma_{\varphi}}$. Interpolating these estimates, we obtain for $\alpha \in (0,1)$ and $\lambda \not\in \overline{\Sigma_{\varphi}}$
		\begin{equation}
			\label{eq:resolvent_estimate}
			\norm{R(\lambda, A)}_{\mathcal{B}(X, X_{\alpha,A})} \le C \abs{\lambda}^{\alpha - 1}.
		\end{equation}
	Concerning duality, if one sets for $\theta \in (0,1)$
		\[
			(\cdot, \cdot)^{'}_{\theta} = \begin{cases}
				[\cdot,\cdot]_{\theta} & \text{if } (\cdot,\cdot)_{\theta} = [\cdot,\cdot]_{\theta}, \\
				(\cdot,\cdot)_{\theta, q'} & \text{if } (\cdot,\cdot)_{\theta} = (\cdot,\cdot)_{\theta, q},
			\end{cases}
		\]
	where $q'$ denotes the Hölder conjugate of $q$, and let $(X'_{\alpha,A'}, A'_{\alpha})_{\alpha \in \IR}$ be the scale for the functor $(\cdot,\cdot)'_{\theta}$, then $X_{\alpha,A}$ is reflexive for all $\alpha \in \IR$, $(X_{\alpha,A})' \simeq X'_{-\alpha,A'}$ and $(A_{\alpha})' = A'_{-\alpha}$ with respect to the duality pairing induced by $\langle \cdot, \cdot \rangle_{X,X'}$ \cite[Theorem~V.1.5.12]{Ama95}.

\subsection{Acquistapace--Terreni condition}\label{sec:at}

	In this section we present the connection between the Acquistapace--Terreni condition and maximal regularity. This is well-known and can be found in~\cite{HieMon00b}. However, the literature on this topic tends to be very sketchy and we feel that the reader can benefit from a complete presentation.
			
	For a family of sectorial operators $(A(t))_{t \in [0,T]}$ on some Banach space $X$ we denote by $L^{\infty}([0,T];D(A(t)))$  the space of all measurable functions $f\colon [0,T] \to X$ for which $f(t) \in D(A(t))$ for almost all $t \in [0,T]$ and $A(\cdot)f(\cdot) \in L^{\infty}([0,T];X)$. Then the following density result holds.
		
	\begin{lemma}\label{lem:inhom_approximation}
		Let $(A(t))_{t \in [0,T]}$ be a family of sectorial operators on some Banach space $X$ for which there exist $\varphi \in (0, \frac{\pi}{2})$ and $C > 0$ with
			\[
				\sup_{t \in [0,T]} \sup_{\lambda \not\in \overline{\Sigma_{\varphi}}} \norm{\lambda R(\lambda, A(t))} \le C.
			\]
		If $t \mapsto R(\lambda, A(t))$ is strongly measurable for all $\lambda < 0$, then $L^{\infty}([0,T]; D(A(t)))$ is dense in $L^p([0,T];X)$ for all $p \in [1, \infty)$.
	\end{lemma}
	\begin{proof}
		It suffices to show that functions in $L^{\infty}([0,T];X)$ can be approximated by elements of $L^{\infty}([0,T];D(A(t)))$ in $L^p([0,T];X)$. Recall that for a sectorial operator $A$ one has $-nR(-n,A)x \to x$ as $n \to \infty$ for all $x \in X$. Now, let $f \in L^{\infty}([0,T];X)$ and set $f_n(t) = -nR(-n,A(t))f(t)$ for $n \in \IN$. Then for all $n \in \IN$ the function $f_n\colon [0,T] \to X$ is measurable and satisfies $f_n(t) \in D(A(t))$ almost everywhere as well as
		\[
			\norm{A(t)f_n(t)}_X \le (1+C)n \norm{f(t)}_X.
		\]
		Moreover, one has $f_n(t) \to f(t)$ and $\norm{f_n(t)}_X \le C \norm{f(t)}_X$ almost everywhere. Hence, the convergence $f_n \to f$ in $L^p([0,T];X)$ holds by dominated convergence.
	\end{proof}
	
	Now suppose that $(A(t))_{t \in [0,T]}$ is a family of sectorial operators on some Banach space $X$ satisfying the following conditions.
	\begin{thm_enum}
		\item For some $\varphi \in (0,\frac{\pi}{2})$ and $C \ge 0$ one has $\sigma(A(t)) \subset \overline{\Sigma_{\varphi}}$ for all $t \in [0,T]$ and
			\begin{equation}
				\label{eq:uniform_sectorial}
				\sup_{t \in [0,T]} \sup_{\lambda \not\in \overline{\Sigma_{\varphi}}} \norm{(1 + \abs{\lambda}) R(\lambda, A(t))} \le C.
			\end{equation}
		\item There exist constants $0 \le \gamma < \beta \le 1$ and $K \ge 0$ such that for all $t,s \in [0,T]$ and all $\lambda \not\in \overline{\Sigma_{\varphi}}$
			\begin{equation*}
				\label{eq:at}
				\tag{AT}
				\normalnorm{A(t)R(\lambda,A(t))(A(t)^{-1} - A(s)^{-1})}_{\mathcal{B}(X)} \le K \frac{\abs{t-s}^{\beta}}{1+\abs{\lambda}^{1-\gamma}}.
			\end{equation*}
	\end{thm_enum}
	The resolvent estimate~\eqref{eq:at} goes back to the work~\cite{AcqTer87}. Not indicating the involved parameters, we therefore call the operator
	\[
		L = A(t)R(\lambda, A(t))(A(t)^{-1} - A(s)^{-1})
	\]
	the \emph{Acquistapace--Terreni} operator for the parameters $t,s \in [0,T]$. Note that $A(t)R(\lambda, A(t))$ is a bounded operator on $X$ and consequently one has $L \in \mathcal{B}(X)$. Let us discuss some immediate consequences of the above assumptions. First, \eqref{eq:uniform_sectorial} implies $0 \in \rho(A(t))$ for all $t \in [0,T]$ together with the uniform estimate $\normalnorm{A^{-1}(t)} \le C$. Applying the operator $(\lambda - A(t)) A(t)^{-1} = \lambda A^{-1}(t) - \Id$ to the left hand side of~\eqref{eq:at} for some $\lambda \not\in \overline{\Sigma}_{\varphi}$ we get
	\[
		\normalnorm{A(t)^{-1} - A(s)^{-1}} \lesssim \abs{t-s}^{\beta}.
	\]
	A fortiori, $t \mapsto A(t)^{-1}$ is measurable. Now, using the analytic expansion of the resolvent~\cite[Chapter~IV, Proposition~1.3(i)]{EngNag00}
	\[
		R(\lambda, A(t)) = -\sum_{n=0}^{\infty} A(t)^{-(n+1)} \lambda^n,
	\]
	which holds for all $\abs{\lambda} < C^{-1}$ independently of $t$, we see that $t \mapsto R(\lambda, A(t))$ is measurable for all $\abs{\lambda} < C^{-1}$. Proceeding with this argument for different balls in the complex plane, we get the measurability of $(t, \lambda) \mapsto R(\lambda, A(t))$. 
	
	Now, let $f \in L^{\infty}([0,T];D(A(t)))$ and $u_0 \in X$. Unwinding the definitions used in the formulation of~\cite[Theorem~6.6]{AcqTer87}, we obtain that~\eqref{nacp} has a classical solution $u$ that satisfies
	\begin{align*}
		\MoveEqLeft A(t)u(t) = \int_0^t A(t)^2 e^{-(t-s) A(t)} (A(s)^{-1} - A(t)^{-1}) A(s) u(s) \d s \\
		& + \int_0^t A(t) e^{-(t-s) A(t)} f(s) \d s + A(t) e^{-t A(t)}u_0.
	\end{align*}
	Introducing formally operators $R,S,Q$ whose mapping properties we study soon as
	\begin{align*}
		(Ru_0)(t) & = A(t)e^{-tA(t)}u_0, \\
		(Sf)(t) & = \int_0^t A(t)e^{-(t-s) A(t)} f(s) \d s, \\
		(Qg)(t) & = \int_0^t A(t)^2 e^{-(t-s) A(t)} (A(s)^{-1} - A(t)^{-1}) g(s) \d s,
	\end{align*} 
	the above equation becomes
	\begin{equation*}
		(\Id - Q)(A(\cdot)u(\cdot))(t) = (Sf)(t) + (Ru_0)(t).
	\end{equation*}
	We first deal with the operator $Q$ and show that $\Id - Q\colon L^p([0,T];X) \to L^p([0;T];X)$ is invertible for $p \in [1,\infty)$ provided the constant in~\eqref{eq:at} is small enough. 
	
	\begin{proposition}\label{prop:invertible}
		Let $p \in [1, \infty)$ and $(A(t))_{t \in [0,T]}$ be a family of sectorial operators satisfying~\eqref{eq:uniform_sectorial}. Suppose additionally that $(A(t))_{t \in [0,T]}$ satisfies~\eqref{eq:at} for some sufficiently small $K > 0$. Then $\Id - Q\colon L^p([0,T];X) \to L^p([0,T];X)$ is invertible. 
	\end{proposition}
	\begin{proof}
		We show that $\norm{Q} < 1$ provided $K$ is small enough. The assertion then follows from the Neumann series expansion. Choose $\psi \in (\varphi, \frac{\pi}{2})$. Then for $\Gamma = \partial \Sigma_{\psi}$, $g \in L^p([0,T];X)$ and $t \in [0,T]$ we may write
			\begin{align*}
				\MoveEqLeft \int_0^t A(t)^2 e^{-(t-s)A(t)} (A(s)^{-1} - A(t)^{-1}) g(s) \d s \\
				& = \frac{1}{2\pi i} \int_0^t A(t) \int_{\Gamma} ze^{-(t-s)z} R(z,A(t)) \d z \, (A(s)^{-1} - A(t)^{-1}) g(s) \d s
			\end{align*}
		Hence, using~\eqref{eq:at}, the norm of the above expression is bounded by
			\begin{align*}
				\MoveEqLeft \frac{K}{2\pi} \int_{0}^t \abs{t-s}^{\beta} \norm{g(s)} \int_{\Gamma} \frac{\abs{z}}{1+\abs{z}^{1-\gamma}} e^{-(t-s) \Re z} \d\abs{z} \d s \\
				& \le \frac{K}{2\pi} \int_0^t \abs{t-s}^{\beta} \norm{g(s)} \int_{\Gamma} \abs{z}^{\gamma} e^{-(t-s) \Re z} \, d\abs{z} \, ds \\
				& = \frac{C K}{2\pi} \int_0^t \abs{t-s}^{\beta - \gamma - 1} \norm{g(s)} \, \d s,
			\end{align*}
		where $C$ is a universal constant only depending on $\gamma$ and $\Gamma$. Consequently, it follows from Minkowski's inequality for convolutions that
			\begin{align*}
				\MoveEqLeft \biggl( \int_0^T \norm{(Qg)(t)}^p \d t \biggr)^{1/p} \le \frac{CK}{2\pi} \normalnorm{(\abs{\cdot}^{\beta - \gamma - 1} \mathds{1}_{[0,T]}) * (\norm{g} \mathds{1}_{[0,T]})}_{L^p(\IR)} \\
				& \le \frac{CK}{2\pi} \biggr( \int_0^{T} t^{\beta - \gamma - 1} \d t \biggr) \norm{g}_{L^p([0,T];X)}.
			\end{align*}
		Note that the integral in brackets is finite because of $\beta > \gamma$. 
	\end{proof}
	
	Suppose now additionally that $\Id - Q\colon L^p([0,T];X) \to L^p([0;T];X)$ is invertible and that $Sf + Ru_0$ lies in $L^p([0,T];X)$ for some $p \in (1, \infty)$. Taking the inverse, we get
	\[
		A(\cdot)u(\cdot) = (\Id - Q)^{-1}(Sf + Ru_0) \in L^p([0,T];X).
	\]
	Using that $u$ is a solution of~\eqref{nacp}, $f \in L^p([0,T];X)$ gives that $\dot{u} \in L^p([0,T];X)$ as well. Hence, $u$ lies in the maximal regularity space for $(A(t))_{t \in [0,T]}$. Further, if $S \in \mathcal{B}(L^p([0,T];X))$ and if there exists a Banach space $Z$ continuously embedded in $X$ such that $R\colon Z \to L^p([0,T];X)$ is bounded, then for some constant $C \ge 0$ independent of the inhomogeneity and of the initial value in $Z$ the solution satisfies the maximal regularity estimate
	\begin{equation}
		\label{eq:general_mr_estimate}
		\norm{u}_{W^{1,p}([0,T];X)} + \norm{A(\cdot)u(\cdot)}_{L^p([0,T];X)} \le C (\norm{f}_{L^p([0,T];X)} + \norm{u_0}_Z).
	\end{equation}
	The case of general inhomogeneities $f \in L^p([0,T];X)$ follows from an approximation argument: by Lemma~\ref{lem:inhom_approximation} there exists a sequence $(f_n)_{n \in \IN} \subset L^{\infty}([0,T];D(A(t)))$ with $f_n \to f$ in $L^p([0,T];X)$. We denote by $(u_n)_{n \in \IN}$ the corresponding classical solutions of~\eqref{nacp} with initial value $u_0 \in Z$. For $n, m \in \IN$ the difference $u_n - u_m$ solves~\eqref{nacp} for $u_0 = 0$. By the maximal regularity estimate~\eqref{eq:general_mr_estimate} one therefore has
	\[
		\norm{u_n - u_m}_{W^{1,p}([0,T];X)} + \norm{A(\cdot)(u_n - u_m)(\cdot)}_{L^p([0,T];X)} \le C \norm{f_n - f_m}_{L^p([0,T];X)}.
	\]
	Thus $(u_n)_{n \in \IN}$ and $(A(\cdot)u_n(\cdot))_{n \in \IN}$ are Cauchy sequences in $W^{1,p}([0,T];X)$ and $L^p([0,T];X)$ respectively and therefore converge to elements $u \in W^{1,p}([0,T];X)$ and $w \in L^p([0,T];X)$. After passing to subsequences all convergences hold almost everywhere in $t \in [0,T]$. For such $t \in [0,T]$ one has $u(t) \in D(A(t))$ and $A(t)u(t) = w(t)$ by the closedness of $A(t)$. Hence, taking limits yields
	\[
		\dot{u}(t) + A(t)u(t) = \lim_{n \to \infty} \dot{u}_n(t) + A(t)u_n(t) = \lim_{n \to \infty} f_n(t) = f(t).
	\]
	Since $u_n(0) \to u(0)$ by the embedding $W^{1,p}([0,T];X) \hookrightarrow C([0,T];X)$, we see that $u$ is a solution of~\eqref{nacp}. Moreover, note that the maximal regularity estimate~\eqref{eq:general_mr_estimate} passes to the limit. It remains to show the uniqueness of the solution in the maximal regularity space. By linearity it suffices to consider the case $u_0 = 0$ and $f = 0$. Differentiating for such a solution $u$ the function $v(s) = e^{-(t-s)A(t)} u(s)$ for fixed $t \in (0,T)$ gives
	 \begin{align*}
	 	\dot{v}(s) & = A(t)e^{-(t-s)A(t)}u(s) + e^{-(t-s)A(t)}\dot{u}(s) \\
		& = A(t)e^{-(t-s)A(t)}u(s) - e^{-(t-s)A(t)}A(s)u(s) \\
		& = A(t)e^{-(t-s)A(t)} (A(s)^{-1} - A(t)^{-1}) A(s) u(s)   
	 \end{align*}
	Integrating over $(0,t)$ and applying $A(t)$ at both sides gives
	\[
		A(t)u(t) = \int_0^t A(t)^2 e^{-(t-s)A(t)} (A(s)^{-1} - A(t)^{-1}) A(s) u(s) \d s = Q(A(\cdot)u(\cdot))(t).
	\]
	Since $\Id - Q$ is invertible, we obtain $A(\cdot)u(\cdot) = 0$ and consequently $u = 0$. Altogether we obtain maximal regularity for the problem~\eqref{nacp}. Hence, everything boils down to the mapping properties of the operators $S$ and $R$. We will deal with $R$ in the context of Banach scales in the next section.
	
\subsubsection{Boundedness of \texorpdfstring{$S$}{S}}

	As done in~\cite{HieMon00b} one can reduce the boundedness of $S$ to the study of a pseudodifferential operator with an operator-valued symbol. In fact, for $f \in C_c^{\infty}([0,T])$ one has
	\[
		(Sf)(t) = \int_{-\infty}^{\infty} A(t)e^{-(t-s) A(t)} \mathds{1}_{[0,\infty)}(t-s)  f(s) \d s.
	\]
	Extend $t \mapsto A(t)$ with $A(0)$ and $A(T)$ at the left and right end of $[0,T]$. Rewriting the above integral gives
	\begin{align*}
		\MoveEqLeft \int_{-\infty}^{\infty} A(t) e^{-(t-s) A(t)} \mathds{1}_{[0,\infty)}(t-s) f(s) \d s \\
		& = \int_{-\infty}^{\infty} A(t) e^{-(t-s) A(t)} \mathds{1}_{[0,\infty)}(t-s) \int_{-\infty}^{\infty} \hat{f}(\xi) e^{2\pi i s \xi} \d\xi \d s \\
		& = \int_{-\infty}^{\infty} A(t) \int_{-\infty}^{\infty} e^{-(t-s)A(t)} \mathds{1}_{[0,\infty)}(t-s) e^{2\pi i s \xi} \d s \, \hat{f}(\xi) \d\xi \\
		& = \int_{-\infty}^{\infty} A(t) \int_{0}^{\infty} e^{-s A(t)} e^{-2\pi i s \xi} \d s \, e^{2\pi i t \xi} \hat{f}(\xi) \d\xi \\
		& = -\int_{-\infty}^{\infty} A(t) R(2\pi i \xi, A(t)) e^{2\pi i t \xi} \hat{f}(\xi) \d\xi.
	\end{align*}
	Hence, the boundedness of $S$ follows from the boundedness of the vector-valued pseudodifferential operator
	\[
		(\hat{S}f)(t) \coloneqq \int_{-\infty}^{\infty} a(t,\xi) \hat{f}(\xi) e^{2\pi i t \xi} \d\xi 
	\]
	shown in \cite{PorStr06} (see also~\cite{HytPor08}), where $a\colon \IR \times \IR \to \mathcal{B}(X)$ denotes the symbol
	\[
		a(t,\xi) = \begin{cases}
			i\xi R(i\xi, A(0)) & t < 0 \\ 
			i\xi R(i\xi, A(t)) & t \in [0,T] \\
			i\xi R(i\xi, A(T)) & t > T.
		\end{cases} 
	\]
	The actual formulation of the theorem uses concepts from vector-valued harmonic analysis to be introduced next.
	
\subsection{Vector-valued harmonic analysis}

	First recall that the Hilbert transform
	\[
		(Hf)(t) \coloneqq \frac{1}{\pi} \lim_{\varepsilon \downarrow 0} \int_{\abs{x} \ge \epsilon} \frac{f(x-t)}{t} \d t
	\]
	initially defined for functions $f \in \mathcal{S}(\IR)$ extends to a bounded operator $H\colon L^p(\IR) \to L^p(\IR)$ for all $p \in (1, \infty)$. A Banach space $X$ is called a \emph{UMD space} if the operator $H \otimes \Id_X$ extends to a bounded operator on $L^p(\IR;X)$ for one or equivalently all $p \in (1, \infty)$. It is easy to see from this definition that $\sigma$-finite $L^p$-spaces for $p \in (1, \infty)$ are UMD spaces. Note that UMD spaces are precisely those spaces on which the most fundamental Fourier multipliers, indicator functions of intervals in $\IR$, define bounded operators on $L^p(\IR;X)$. Hence, only on these spaces a rich theory can be developed. For further details on UMD spaces we refer to~\cite{Fra86} and~\cite{Bur01}.
	
	Further, it is by now well-known that vector-valued analogues of classical theorems in harmonic analysis, e.g.\ the Mikhlin multiplier theorem, only hold in the non-Hilbert space setting if boundedness in operator norm is replaced by a stronger boundedness concept called $\mathcal{R}$-boundedness. A subset $\mathcal{T} \subset \mathcal{B}(X)$ for some Banach space $X$ is called \emph{$\mathcal{R}$-bounded} if there exists a constant $C \ge 0$ such that for all $n \in \IN$, $T_1, \ldots, T_n \in \mathcal{T}$ and $x_1, \ldots, x_n \in X$ one has
	\[
		\IE \biggnorm{\sum_{k=1}^n \epsilon_k T_k x_k} \le C \IE \biggnorm{\sum_{k=1}^n \epsilon_k x_k},
	\]
	where $\epsilon_1, \ldots, \epsilon_n$ are $n$ independent identically distributed Rademacher variables, i.e.\ $\IP(\epsilon_k = \pm 1) = \frac{1}{2}$ for all $k = 1, \ldots, n$. The smallest constant $C \ge 0$ for which the above inequality holds is denoted by $\mathcal{R}(\mathcal{T})$. Note that this implies $\mathcal{R}(\mathcal{T} \mathcal{S}) \le \mathcal{R}(\mathcal{T}) \mathcal{R}(\mathcal{S})$ for $\mathcal{T}, \mathcal{S} \subset \mathcal{B}(X)$. Furthermore, it follows from Kahane's contraction principle~\cite[Proposition~2.5]{KunWei04} that $\mathcal{R} \{ \lambda \Id_X: \lambda \in [-1,1] \} = 1$.

\section{Maximal Regularity via Acquistapace--Terreni and Banach Scales}

	We now present a new variant of the known Acquistapace--Terreni result for maximal regularity in terms of Banach scales that is motivated by the methods used in~\cite{OuhSpi10} in the Hilbert space setting.
	
\subsection{Initial value zero}
		
	We start with the case of initial value zero.
	
	\begin{proposition}\label{prop:mr_zero}
		For $T \in (0, \infty)$ let $(A(t))_{t \in [0,T]}$ be a family of sectorial operators on some UMD Banach space $X$. Suppose that $(A(t))_{t \in [0,T]}$ satisfies the following assumptions.
		\begin{thm_enum}
			\item \label{thm:abstract_mr_zero:r_boundedness} For some $\varphi \in (0,\frac{\pi}{2})$ one has $\sigma(A(t)) \subset \overline{\Sigma_{\varphi}}$ for all $t \in [0,T]$ and
				\[
					\R{(1 + \abs{\lambda}) R(\lambda, A(t)): \lambda \not\in \overline{\Sigma_{\varphi}}, t \in [0,T]} < \infty.
				\]
			\item\label{thm:abstract_mr_zero:hoelder} There exist constants $0 \le \gamma < \beta \le 1$ and $C \ge 0$ such that for all $t,s \in [0,T]$ one has
				\begin{equation}
					\label{eq:hölder_estimate_zero}
					\abs{\langle (A_{-1}(t) - A_{-1}(s))x, x' \rangle} \le C \abs{t-s}^{\beta} \norm{x}_{D(A(s))} \norm{x'}_{X'_{\gamma, A(t)'}}
				\end{equation}
				for all $x \in D(A(s))$ and all $x' \in D(A(t)')$.
		\end{thm_enum}
		Then the non-autonomous Cauchy problem \eqref{nacp} for $(A(t))_{t \in [0,T]}$ has maximal $L^p$-regularity for $u_0 = 0$ and all $p \in (1, \infty)$.
	\end{proposition}
	
	Before giving a proof of this theorem, let us comment on the various assumptions. Concerning assumption~\ref{thm:abstract_mr_zero:hoelder}, notice that $A_{-1}(t)x \in X_{-1,A(t)}$ for all $x \in X$ and therefore the duality is well-defined for all $x' \in X'_{1,A(t)'} = D(A(t)')$. Assumption~\ref{thm:abstract_mr_zero:hoelder} is probably the most restrictive. The limitation to Hölder regularity stems from our approach via the Acquistapace--Terreni representation formula for solutions of~\eqref{nacp}. Nevertheless, the recent negative solution to Lions' problem shows that in the Hölder scale the required regularity is optimal for general operators~\cite{Fac16c}.
		
	Assumption~\ref{thm:abstract_mr_zero:r_boundedness} is rather natural. It reduces to uniform boundedness in the Hilbert space case and therefore to the result proved in~\cite{HieMon00b}. In the non-Hilbert space case boundedness in operator norm cannot be sufficient for maximal regularity as the existence of counterexamples shows already in the autonomous case (see~\cite{KalLan00}, \cite{Fac14} and~\cite[Theorem~3.18]{Fac14c} for a self-contained counterexample on $L^p$-spaces). In Section~\ref{sec:r_boundedness} we will verify assumption~\ref{thm:abstract_mr_zero:r_boundedness} for a broad class of second order elliptic operators in divergence form.
	
	In the following proofs we will often shift the involved operators. This does not influence maximal regularity as the next  easy to prove lemma shows.
	
	\begin{lemma}\label{lem:mr_shift}
		Let $(A(t))_{t \in [0,T]}$ be a family of closed operators for which the problem~\eqref{nacp} has maximal $L^p$-regularity. Then $(A(t) + \mu)_{t \in [0,T]}$ has maximal $L^p$-regularity for all $\mu \in \IR$ and the same space of initial values.
	\end{lemma}

	We are now ready to give the proof of Proposition~\ref{prop:mr_zero}.
	
	\begin{proof}[Proof of Proposition~\ref{prop:mr_zero}]
		We first treat the Acquistapace--Terreni estimate~\eqref{eq:at}. For this let $x \in X$ and fix $s, t \in [0,T]$ as well as $\lambda \not\in \overline{\Sigma_{\varphi}}$. Using the fact that $A_{-1}(t) \in \mathcal{B}(X, X_{-1,A(t)})$ extends the action of $A$, we can rewrite the Acquistapace--Terreni operator as
			\begin{align*}
				Lx & = A(t) R(\lambda, A(t)) (A(t)^{-1}  - A(s)^{-1})x \\
				& = R(\lambda, A_{-1}(t)) A_{-1}(t) (A(t)^{-1} - A(s)^{-1})x \\
				& = R(\lambda, A_{-1}(t)) (A(s) - A_{-1}(t)) A(s)^{-1}x
			\end{align*}
		Testing this identity with $x' \in D(A(t)')$ gives
			\begin{align*}
				\langle Lx, x' \rangle_{X,X'} & = \langle R(\lambda, A_{-1}(t)) (A(s) - A_{-1}(t)) A(s)^{-1}x, x' \rangle_{X_{-1,A(t)},X_{1,A(t)'}} \\
				& = \langle (A(s) - A_{-1}(t)) A(s)^{-1}x, R(\lambda,A_1(t)')x' \rangle_{X_{-1,A(t)},X_{1,A(t)'}}.
			\end{align*}
		By density the above identity extends to all $x' \in X'$. Hence, for all $x' \in X'$ we have
			\begin{align*}
				\langle Lx, x' \rangle & = \langle A_{-1}(s) A(s)^{-1}x, R(\lambda, A(t)') x' \rangle - \langle A_{-1}(t) A(s)^{-1}x, R(\lambda, A(t)') x' \rangle.
			\end{align*}
		By the Hölder continuity assumption \ref{thm:abstract_mr_zero:hoelder} we further get
			\begin{align*}
				\abs{\langle Lx, x' \rangle} & \le C \abs{t-s}^{\beta} \normalnorm{A(s)^{-1}x}_{D(A(s))} \normalnorm{R(\lambda, A(t)')x'}_{X'_{\gamma, A(t)'}}. 
			\end{align*}
		
		Now fix $\mu_0 > 0$. For $\mu \in (\mu_0, \infty)$ assumption~\ref{thm:abstract_mr_zero:hoelder} holds for the shifted family $(A(t) + \mu)_{t \in [0,T]}$ as well, i.e.\ one has
			\begin{equation}
				\label{eq:shifted_hoelder}
				\abs{\langle ((A(t) + \mu)_{-1} - (A(s) + \mu)_{-1})x, x' \rangle} \le C \abs{t-s}^{\beta} \norm{x}_{D(A(s))} \norm{x'}_{X'_{\gamma, A(t)'}}.
			\end{equation}
		Here $(A(t) + \mu)_{-1}$ denotes the extension of $A(t) + \mu$ to $X_{A(t),-1}$ which agrees with $A_{-1}(t) + \mu$. Choose $\epsilon \in (0, \beta - \gamma)$. Then the above deduced estimate applied to $(A(t) + \mu)_{t \in [0,T]}$ gives for some fixed $M \ge 0$ and all $x \in X$, $x' \in X'$ and $\lambda \not\in \overline{\Sigma_{\varphi}}$
			\begin{align*}
				\MoveEqLeft \langle (A(t) + \mu) R(\lambda, A(t) + \mu) ((\mu + A(t))^{-1}  - (\mu + A(s))^{-1})x, x' \rangle \\
				& \le C \abs{t-s}^{\beta} \normalnorm{(\mu + A(s))^{-1}x}_{D(A(s))} \normalnorm{R(\lambda - \mu, A(t)') x'}_{X'_{\gamma,A(t)'}} \\
				& \le CM \abs{t-s}^{\beta} \abs{\lambda - \mu}^{\gamma - 1} \normalnorm{A(s)(\mu + A(s))^{-1}x}_{X} \norm{x'}_{X'},
			\end{align*}
		where we have used estimate~\eqref{eq:resolvent_estimate} in the last inequality. We first remark that $A(s)(\mu + A(s))^{-1} = \Id - \mu(\mu + A(s))^{-1}$ which is uniformly bounded in $\mu \in (\mu_0, \infty)$ and $s \in [0,T]$ as a consequence of assumption~\ref{thm:abstract_mr_zero:r_boundedness}. Secondly, for $\lambda \not\in \overline{\Sigma_{\varphi}}$ and $\mu \in (\mu_0, \infty)$ one has
			\begin{align*}
				\frac{1}{\normalabs{\lambda - \mu}^{1-\gamma}} = \frac{1}{\normalabs{\lambda - \mu}^{\epsilon}} \frac{1}{\normalabs{\lambda - \mu}^{1-(\gamma + \epsilon)}} \le \sup_{\lambda \not\in \overline{\Sigma_{\varphi}}} \frac{1}{\normalabs{\lambda - \mu}^{\epsilon}} \cdot \sup_{\mu \in (\mu_0,\infty)} \frac{1}{\normalabs{\lambda - \mu}^{1-(\gamma + \epsilon)}}.
			\end{align*}
		We now deal with both factors separately. For the first factor on the right hand side we use $\abs{\lambda - \mu} \ge  \mu \sin \varphi$, whereas for the second one has $\abs{\lambda - \mu} \ge \abs{\lambda - \mu_0}$ if $\Re \lambda \le \mu_0$ and $\abs{\lambda - \mu} \ge \abs{\Im \lambda} \ge \sin \varphi \abs{\lambda}$ if $\Re \lambda \ge \mu_0$. Furthermore the quotients
			\begin{align*}
				\frac{1 + \normalabs{\lambda}^{1-(\gamma + \epsilon)}}{\normalabs{\lambda - \mu_0}^{1-(\gamma + \epsilon)}} & \text{ on } \overline{\Sigma_{\varphi}}^c \cap \{ \lambda: \Re \lambda \le \mu_0 \}, \\
				\frac{1 + \normalabs{\lambda}^{1-(\gamma + \epsilon)}}{\normalabs{\lambda}^{1-(\gamma + \epsilon)}} & \text{ on } \overline{\Sigma_{\varphi}}^c \cap \{ \lambda: \Re \lambda \ge \mu_0 \}
			\end{align*}
		are bounded. Altogether we obtain the estimate
		\begin{align*}
			\MoveEqLeft \normalnorm{(A(t) + \mu)R(\lambda, A(t) + \mu) ((A(t) + \mu)^{-1}  - (A(s) + \mu)^{-1})} \\
			& \lesssim \frac{1}{\mu^{\epsilon}} \abs{t-s}^{\beta} \frac{1}{1 + \normalabs{\lambda}^{1-(\gamma + \epsilon)}},
		\end{align*}
		where the implicit constant in the estimate is universal for all variables on the right hand side. Hence, for suitable large $\mu$ the corresponding operator $\Id - Q$ is invertible on $L^p([0,T];X)$ by Proposition~\ref{prop:invertible}. We now fix such a $\mu \in (\mu_0, \infty)$. Note that
		\begin{align*}
			\MoveEqLeft \mathcal{R}\{ (1 + \abs{\lambda}) R(\lambda, A(t) + \mu): \lambda \not\in \overline{\Sigma_{\varphi}}, t \in [0,T] \} \\
			& = \mathcal{R}\{ (1 + \abs{\lambda}) R(\lambda - \mu, A(t)): \lambda \not\in \overline{\Sigma_{\varphi}}, t \in [0,T] \} \\
			& \le \sup_{\lambda \not\in \overline{\Sigma_{\varphi}}} \frac{1 + \abs{\lambda}}{1 + \abs{\lambda - \mu}} \cdot \mathcal{R}\{ (1+\abs{\lambda}) R(\lambda, A(t)): \lambda \not\in \overline{\Sigma_{\varphi}}, t \in [0,T] \},
		\end{align*}
		where we have used Kahane's contraction principle in the last inequality. Since the supremum is finite, we see that $(A(t) + \mu)_{t \in [0,T]}$ satisfies assumption~\ref{thm:abstract_mr_zero:r_boundedness}. 
		
		It is shown in~\cite[Corollary~14]{PorStr06} that if $\Id - Q$ is invertible as shown above and if assumption~\ref{thm:abstract_mr_zero:r_boundedness} holds as just shown, then the singular integral operator $\hat{S}$ is bounded on $L^p([0,T];X)$ for all $p \in (1, \infty)$. By the considerations at the beginning of this section we have shown maximal regularity for the problem
		\[
			\dot{u}(t) + (A(t) + \mu)u(t) = f(t)
		\]
		with initial value equal to zero. The assertion now follows from Lemma~\ref{lem:mr_shift}.
	\end{proof}
	
\subsection{Non-zero initial value}

	From Section~\ref{sec:at} we know that maximal regularity holds for all initial values in some embedded Banach space $Z$ under the assumptions of Proposition~\ref{prop:mr_zero} provided $R\colon Z \to L^p([0,T];X)$ is bounded. Recall that in the autonomous case, i.e.\ $A(t) = A$, maximal regularity holds precisely for initial values in the real interpolation space $(D(A),X)_{1/p,p}$ (\cite[Lemma~1]{ChiFio14} and~\cite[Corollary 1.14]{Lun09}). Hence, this is the best we can expect in the non-autonomous case and, indeed, this holds automatically by~\cite[Lemma~2.1]{GioLunSch05}.	However, we prefer to give a self-contained more streamlined proof. In the following, as an exception to the rule, the space $D(A(0))$ is endowed with the graph norm whenever it is used in the context of initial values for~\eqref{nacp}.
	
	\begin{theorem}\label{thm:abstract_mr_non-zero}
		For $T \in (0, \infty)$ let $(A(t))_{t \in [0,T]}$ be a family of sectorial operators on some UMD Banach space $X$. Suppose that $(A(t))_{t \in [0,T]}$ satisfies the following assumptions.
		\begin{thm_enum}
			\item \label{thm:abstract_mr_non-zero:r_boundedness} For some $\varphi \in (0,\frac{\pi}{2})$ one has $\sigma(A(t)) \subset \overline{\Sigma_{\varphi}}$ for all $t \in [0,T]$ and
				\[
					\R{ (1 + \abs{\lambda}) R(\lambda, A(t)): \lambda \not\in \overline{\Sigma_{\varphi}}, t \in [0,T]} < \infty.
				\]
			\item\label{thm:abstract_mr_non-zero:hoelder} There exist constants $0 \le \gamma < \beta \le 1$ and $C \ge 0$ such that for all $t,s \in [0,T]$ one has
				\begin{equation}
					\label{eq:hölder_estimate_non-zero}
					\abs{\langle (A_{-1}(t) - A_{-1}(s))x, x' \rangle} \le C \abs{t-s}^{\beta} \norm{x}_{D(A(s))} \norm{x'}_{X'_{\gamma,A(t)'}}
				\end{equation}
				for all $x \in D(A(s))$ and all $x' \in D(A(t)')$.
		\end{thm_enum}
		Then for all $p \in (1, \infty)$ and $\mu \in \IR$ the non-autonomous Cauchy problem \eqref{nacp} for $(A(t) + \mu)_{t \in [0,T]}$ has maximal $L^p$-regularity for all $u_0 \in (D(A(0)+\mu),X)_{1/p,p}$.
	\end{theorem}
	
	\begin{proof}
		In the following we use a perturbation argument. We consider the the operator
		\[
			(R_0u_0)(t) = A(0) e^{-tA(0)}u_0,
		\]
		which is bounded from $(D(A(0)),X)_{1/p,p}$ to $L^p([0,T];X)$. By the arguments already given in the proof of Proposition~\ref{prop:mr_zero}, it suffices to establish the boundedness of $R - R_0\colon (D(A(0)),X)_{1/p,p} \to L^p([0,T];X)$ if $(A(t))_{t \in [0,T]}$ is replaced by $(A(t) + \mu)_{t \in [0,T]}$ for some fixed $\mu > 0$ large enough. Indeed, as seen in Proposition~\ref{prop:mr_zero}, one obtains the invertibility of $Q$ and the boundedness of $S$ for such $\mu$. Furthermore, we have seen that $(A(t) + \mu)_{t \in [0,T]}$ satisfies the assumptions of this theorem as well as inequality~\eqref{eq:shifted_hoelder}. Therefore the following calculations remain valid if one replaces $(A(t))_{t \in [0,T]}$ by $(A(t) + \mu)_{t \in [0,T]}$ and consequently one obtains maximal regularity for the family $(A(t) + \mu)_{t \in [0,T]}$. Once maximal regularity for $(A(t) + \mu)_{t \in [0,T]}$ is established for some $\mu > 0$, maximal regularity follows for all $\mu \in \IR$ by Lemma~\ref{lem:mr_shift}. Here we use the equivalence between the graph norms of $A(0)$ and $A(0) + \mu$ which passes to the interpolation spaces.
		
		First recall that $A(t) e^{-tA(t)}$ are bounded operators on $X$ because $-A(t)$ generate analytic semigroups. Now choose $\psi \in (\varphi, \frac{\pi}{2})$. Then for $x \in X$ and $\Gamma = \partial \Sigma_{\psi}$ we have
			\begin{align*}
				\MoveEqLeft A(t)e^{-tA(t)}x - A(0)e^{-tA(0)}x = \frac{1}{2\pi i} \int_{\Gamma} ze^{-tz} [R(z,A(t)) - R(z,A(0))] x \d z \\
				& = \frac{1}{2\pi i} \int_{\Gamma} ze^{-tz} R(z,A_{-1}(t)) [A_{-1}(t) - A(0)] R(z,A(0))x \d z,
			\end{align*}
		by the second resolvent identity. Testing with $x' \in X'$ gives
			\begin{align*}
				\MoveEqLeft \langle A(t)e^{-tA(t)}x - A(0)e^{-tA(0)}x, x' \rangle \\
				& = \frac{1}{2\pi i} \int_{\Gamma} ze^{-tz} \langle R(z,A_{-1}(t)) [A_{-1}(t) - A(0)] R(z,A(0))x, x' \rangle \d z \\
				& = \frac{1}{2\pi i} \int_{\Gamma} ze^{-tz} \langle (A_{-1}(t) - A(0)) R(z,A(0))x, R(z,A(t)')x' \rangle \d z.
			\end{align*}
		One has the resolvent estimates
		\[
			\norm{R(z,A(0))}_{\mathcal{B}(X,D(A(0)))} \lesssim 1,  \qquad \norm{R(z,A(0))}_{\mathcal{B}(D(A(0)),D(A(0)))} \lesssim \abs{z}^{-1}.
		\] 
		Using real interpolation we see that for $\theta \in [0,1]$ one has
		\[
			\norm{R(z,A(0))}_{\mathcal{B}((D(A(0)),X)_{1/p,p},D(A(0)))} \lesssim \normalabs{z}^{1/p - 1}.
		\] 
		Using assumption~\ref{thm:abstract_mr_non-zero:hoelder} together with this estimate, for $x \in (D(A(0)),X)_{1/p,p}$ we obtain the pointwise estimate
			\begin{align*}
				\MoveEqLeft \normalabs{\langle A(t)e^{-tA(t)}x - A(0)e^{-tA(0)}x, x' \rangle} \\
				& \le \frac{C}{2\pi} t^{\beta} \int_{\Gamma} \abs{z} e^{-t \Re z} \norm{R(z,A(0))x}_{D(A(0))} \norm{R(z,A(t)')x'}_{X'_{\gamma,A(t)'}} \d\abs{z} \\
				& \lesssim t^{\beta} \int_{\Gamma} \abs{z}^{1/p + \gamma - 1} e^{-t \Re z} \d\abs{z} \norm{x}_{(D(A(0)),X)_{1/p,p}} \norm{x'}_{X'} \\
				& \lesssim t^{\beta - 1/p - \gamma} \norm{x}_{(D(A(0)),X)_{1/p,p}} \norm{x'}_{X'}.
			\end{align*}
		Integrating over the time variable, we see that $R-R_0$ satisfies the estimate
			\begin{align*}
				\norm{(R - R_0)x}_{L^p([0,T];X)} \lesssim \biggl( \int_0^T t^{p(\beta - 1/p - \gamma)} \d t \biggr)^{1/p} \norm{x}_{(D(A(0)),X)_{1/p,p}}.
			\end{align*}
		Now, the assumption $\beta > \gamma$ guarantees that the above integral is finite and therefore the boundedness of $R\colon (D(A(0)),X)_{1/p,p} \to L^p([0,T];X)$ is shown.
	\end{proof}

	\begin{remark} 
		Note that in the general setting of~Theorem~\ref{thm:abstract_mr_non-zero} we cannot apply the extrapolation results for maximal regularity proved in~\cite[Corollary~3]{ChiFio14} and generalized in~\cite[Theorem~5.3]{ChiKro14} as these require both $(A(t))_{t \in [0,T]}$ and $(A(t)')_{t \in [0,T]}$ to satisfy conditions~\eqref{eq:uniform_sectorial} and~\eqref{eq:at}. However, we will see later that even in concrete applications~\eqref{eq:at} can only be verified for $(A(t))_{t \in [0,T]}$.
	\end{remark}
	
	\begin{remark}\label{rem:comparison_form}
		Let $a\colon [0,T] \times V \times V \to \IR$ be a family of forms for which $t \mapsto a(t,u,v)$ is measurable for all $u, v \in V$. Moreover, we require that there exist positive constants $\alpha_0$ and $M$ such that for all $t \in [0,T]$ and all $u,v \in V$ one has
		\begin{align*}
			\abs{a(t,u,v)} \le M \norm{u}_V \norm{v}_V, \qquad a(t,u,u) \ge \alpha_0 \norm{u}_V^2.
		\end{align*}
		Maximal regularity for forms then asks under which conditions on the forms and the initial data one has maximal regularity for the family $(A(t))_{t \in [0,T]}$, where $A(t)$ is the operator associated to $a(t, \cdot, \cdot)$. Note that there is a Banach scale adapted to this problem. For this we set
		\begin{equation*}
			H_{\alpha,A(t)} =
			\begin{cases}
				[H,V]_{2\alpha} & \text{if } \alpha \in [0, \frac{1}{2}] \\
			 	[V,D(A(t))]_{2(\alpha - \frac{1}{2})} & \text{if } \alpha \in [\frac{1}{2},1]. \\
			\end{cases}
		\end{equation*}
		Analogously, one interpolates between the spaces $H$, $V_{-1,A(t)}$ and $H_{-1,A(t)}$ in the case $\alpha \in [-1,0]$. In the setting of time dependent forms with independent form domains the extrapolated operators $A(t)_{-1/2}$ get stable domains. Furthermore, the scale $H_{\alpha,A(t)}$ is equivalent to the complex interpolation scale associated to the operator $A(t)$ if the form $a$ has the Kato square root property, i.e.\ if $[H,D(A(t))]_{1/2} \simeq D(A(t)^{1/2})$. Although we do not have presented a proof of Theorem~\ref{thm:abstract_mr_non-zero} for abstract Banach scales in order to simplify matters, one can verify that the proof also works for the above Banach scale.
		
		Let us now check the assumptions of Theorem~\ref{thm:abstract_mr_non-zero} in the form setting. First note that~\ref{thm:abstract_mr_non-zero:r_boundedness} is satisfied because our assumptions on the forms guarantee uniform sectorial estimates for all associated operators $(A(t))_{t \in [0,T]}$ and because boundedness in operator norm is equivalent to $\mathcal{R}$-boundedness on Hilbert spaces. Secondly, the Hölder estimate in assumption~\ref{thm:abstract_mr_non-zero:hoelder} is satisfied for $\gamma = \frac{1}{2}$ if for some $C \ge 0$ and $\beta > \frac{1}{2}$ the forms satisfy for all $u, v \in V$
		\[
			\abs{a(t,u,v) - a(s,u,v)} \le C \abs{t-s}^{\beta} \norm{u}_V \norm{v}_V \qquad \text{for all } s, t \in [0,T].
		\]
		This essentially reproduces the results in~\cite{HaaOuh15}. Further, by taking the flexibility of the choice of $H'_{\gamma,A(t)'}$ into account, we also obtain the results proved~\cite{AreMon14} and~\cite{Ouh15}. Under more restrictive assumptions different positive results are known. In fact, if the forms are additionally assumed to be symmetric, then one has maximal regularity if the time dependence is of bounded variation~\cite{Die15}. 
	\end{remark}

	In order to apply Theorem~\ref{thm:abstract_mr_non-zero} one has to deal with two different assumptions: on the one hand one must check the $\mathcal{R}$-boundedness condition and on the other hand the Hölder estimate has to be verified. We deal with the first one in Section~\ref{sec:r_boundedness} and will verify the second one in concrete applications in the sections thereafter.

\section{Assumptions Satisfied for General Elliptic Operators}\label{sec:r_boundedness}

	In this section we give sufficient conditions for assumption~\ref{thm:abstract_mr_non-zero:r_boundedness} of Theorem~\ref{thm:abstract_mr_non-zero} to hold and discuss a further technical issue, namely uniform estimates for bounded imaginary powers as needed in Proposition~\ref{prop:bip_fractional}. For this we need further results on $\mathcal{R}$-boundedness and its generalization $\mathcal{R}_q$-boundedness which are special for $L^p$-spaces (or more general for Banach lattices with finite concavity). Further, Gaussian estimates play a central role for the verification of the needed estimates. 
	
	\subsection{The \texorpdfstring{$\mathcal{R}$}{R}-boundedness assumption}
	
	As a motivation we note that as a consequence of the Khintchine inequality~\cite[Appendix~C]{Gra08} the $\mathcal{R}$-boundedness of a family $\mathcal{T} \subset \mathcal{B}(L^p(\Omega))$ for $p \in [1, \infty)$ is equivalent to the validity of the square function estimate
	\begin{equation*}
		\biggnorm{\biggl( \sum_{k=1}^n \abs{T_kf_k}^2 \biggr)^{1/2}}_p \le C \biggnorm{\biggl( \sum_{k=1}^n \abs{f_k}^2 \biggr)^{1/2}}_p
	\end{equation*}
	for a constant $C \ge 0$ independent of $n \in \IN$, $T_1, \ldots, T_n \in \mathcal{T}$ and $f_1, \ldots, f_n \in L^p(\Omega)$. Using this characterization, one obtains a straightforward generalization by replacing the squares by general exponents.
	
	\begin{definition}
		Let $p, q \in [1, \infty]$ and $(\Omega, \Sigma, \mu)$ be a measure space. A subset $\mathcal{T} \subset \mathcal{B}(L^p(\Omega))$ is called \emph{$\mathcal{R}_q$-bounded} if there exists $C \ge 0$ such that for all $n \in \IN$, $T_1, \ldots, T_n \in \mathcal{T}$ and $f_1, \ldots, f_n \in L^p(\Omega)$ one has
			\begin{align*}
				\biggnorm{\biggl( \sum_{k=1}^n \abs{T_kf_k}^q \biggr)^{1/q}}_p & \le C \biggnorm{\biggl( \sum_{k=1}^n \abs{f_k}^q \biggr)^{1/q}}_p & \text{for } q \in [1, \infty), \\
				\normalnorm{\sup_{k} \abs{T_kf_k}}_p & \le C \normalnorm{\sup_k \abs{f_k}}_p & \text{for } q = \infty.
			\end{align*}
		We denote by $\mathcal{R}_q(\mathcal{T})$ the smallest constant $C$ for which the above inequality holds.
	\end{definition}
	
	Note that $\mathcal{R}_q$-bounded subsets are always bounded in operator norm and that the converse holds in the special case $p = q$. Further, $\mathcal{R}_2$-boundedness is equivalent to $\mathcal{R}$-boundedness. We need the permanence properties of $\mathcal{R}_q$-bounded sets given in the next two lemmata. For the proof of the first one see~\cite[Proposition~3.1.10]{Ull10}.

	\begin{lemma}\label{lem:r_boundedness_l1}
		Let $(\Omega, \Sigma,\mu)$ be a $\sigma$-finite measure space, $p, q \in (1, \infty)$ and $\tau \subset \mathcal{B}(L^p(\Omega))$ an $\mathcal{R}_q$-bounded set. Then
			\begin{equation*}
				\mathcal{T} \coloneqq \left\{ \int_{\Omega} h(\omega) N(\omega) \d\mu(\omega): \norm{h}_{L^1(\Omega)} \le 1, N\colon \Omega \to \tau \text{ measurable} \right\}
			\end{equation*}
		is $\mathcal{R}_q$-bounded with $\mathcal{R}_q(\mathcal{T}) \le 2 \mathcal{R}_q(\tau)$.
	\end{lemma}
	
	This result follows from the fact that the elements of $\mathcal{T}$ lie in the closure of the absolute convex hull of $\tau$ in the strong operator topology which is $\mathcal{R}_q$-bounded with bound at most $2 \mathcal{R}_q(\tau)$. The next result is a consequence of Lemma~\ref{lem:r_boundedness_l1} and the Poisson integral formula (compare with the proof of~\cite[Example~2.16]{KunWei04}).
	
	\begin{lemma}\label{lem:r_boundedness_analytic}
		Let $p, q \in (1, \infty)$ and $I$ be an arbitrary index set. Further suppose that for $i \in I$ one has a family of analytic functions $\Sigma_{\theta'} \ni \lambda \mapsto N_i(\lambda) \in \mathcal{B}(L^p(\Omega))$ for $\theta' \in (0, \pi)$ satisfying
			\[
				\tau \coloneqq \R[q]{N_i(\lambda): \lambda \in \partial \Sigma_{\theta}, \lambda \neq 0, i \in I} < \infty.
			\]
		for some $\theta \in (0,\theta')$. Then
			\[
				\R[q]{N_i(\lambda): \lambda \in \Sigma_{\theta}, i \in I} \le 2 \mathcal{R}_q(\tau).
			\]
	\end{lemma}
	
	We are now ready to prove the following domination result for families of sectorial operators.
	
	\begin{proposition}\label{prop:r_boundedness}
		Let $(\Omega, \Sigma, \mu)$ be a $\sigma$-finite measure space and $p \in (1, \infty)$. Suppose that $(A(t))_{t \in [0,T]}$ is a family of sectorial operators on $L^p(\Omega, \Sigma, \mu)$ with the following properties. 
		\begin{thm_enum}
			\item\label{prop:r_boundedness:uniform_boundedness} There exists $\varphi \in (0,\frac{\pi}{2})$ such that $\sigma(A(t)) \subset \overline{\Sigma_{\varphi}}$ for all $t \in [0,T]$ and
				\[
					\sup_{t \in [0,T]} \sup_{\lambda \not\in \overline{\Sigma_{\varphi}}} \norm{\lambda R(\lambda, A(t))} < \infty.
				\]
			\item\label{prop:r_boundedness:domination} For $p \neq q$ there exists an $\mathcal{R}_q$-bounded family $\mathcal{T}$ dominating the semigroups generated by $-A(t)$ in the following sense: for all $t \in [0,T]$ and all $s \ge 0$ there exists $T \in \mathcal{T}$ such that for all $f \in L^p(\Omega)$ one has $\normalabs{e^{-sA(t)}f}(x) \le (T \abs{f})(x)$ almost everywhere.
		\end{thm_enum} 
		 Then for all $r \in (p,q)$ respectively $r \in (q,p)$ there exists some $\theta \in (\varphi, \frac{\pi}{2})$ with
			\[
				\R[r]{\lambda R(\lambda, A(t)): \lambda \not\in \overline{\Sigma_{\theta}}, t \in [0,T]} < \infty.
			\]
		In particular, if the interval is such that $2 \in (p,q)$ respectively $2 \in (q,p)$, then
			\[
				\R{\lambda R(\lambda, A(t)): \lambda \not\in \overline{\Sigma_{\theta}}, t \in [0,T]} < \infty.
			\]
	\end{proposition}
	\begin{proof}
		It follows from the domination assumption~\ref{prop:r_boundedness:domination} that the set $\{ e^{-sA(t)}: s \ge 0, t \in [0,T] \}$ is $\mathcal{R}_q$-bounded. Now, recall that in the strong operator topology
			\[
				\lambda R(\lambda, -A(t)) = \int_0^{\infty} \lambda e^{-\lambda t} e^{-s A(t)} \d s
			\]
		for all $\Re \lambda > 0$. This representation applied to Lemma~\ref{lem:r_boundedness_l1} shows that for all $\theta' \in (\frac{\pi}{2}, \pi)$ the set $\{ \lambda R(\lambda, A(t)): \lambda \not\in \overline{\Sigma_{\theta'}}, t \in [0,T] \}$ is $\mathcal{R}_q$-bounded. Further, it follows from the uniform boundedness assumption~\ref{prop:r_boundedness:uniform_boundedness} that $\{ \lambda R(\lambda, A(t)): \lambda \not\in \overline{\Sigma_{\varphi}}, t \in [0,T] \}$ is $\mathcal{R}_p$-bounded. We now employ an interpolation argument. Let $\theta_0 > \varphi$ and $\theta_1 > \frac{\pi}{2}$. For $n \in \IN$ fix $t_1, \ldots, t_n \in [0,T]$ and $s_1, \ldots, s_n > 0$. Now, for $\lambda$ in the strip $S = \{ \lambda \in \IC: \theta_0 \le \Re \lambda \le \theta_1 \}$ consider the linear operators
			\begin{align*}
				M(\lambda)\colon L^p(\Omega; \ell^{\infty}_n) & \to L^p(\Omega; \ell^{\infty}_n) \\
				(f_1, \ldots, f_n) & \mapsto (s_1 e^{i\lambda} R(s_1 e^{i\lambda}, A(t_1))f_1, \ldots, s_1 e^{i\lambda} R(s_n e^{i \lambda}, A(t_n))f_n).
			\end{align*}
		depending analytically on $\lambda$. Reformulating the findings of the first part of the proof, we see that
			\[
				M(\theta_0 + iy)\colon L^p(\Omega; \ell^p_n) \to L^p(\Omega; \ell^p_n) \quad \text{and} \quad M(\theta_1 + iy)\colon L^p(\Omega; \ell^q_n) \to L^p(\Omega; \ell^q_n)
			\]
		are bounded for $y \in \IR$ with $\norm{M(\theta_j + i y)} \le C$ for $j = 0, 1$ and some $C \ge 0$. As a consequence of the abstract Stein interpolation theorem~\cite[Theorem~2.1]{Voi92} we obtain for all $\alpha \in (0,1)$ that the operator $M(\theta)$ is a bounded operator in $L^p(\Omega; \ell^r_n)$ with $\norm{M(\theta)} \le C$, where $\theta \in (\theta_0, \theta_1)$ and $r$ satisfy
			\[
				\theta = (1-\alpha) \theta_0 + \alpha \theta_1 \qquad \text{and} \qquad \frac{1}{r} = (1-\alpha) \frac{1}{p} + \alpha \frac{1}{q}.
			\]
		Note that the constant $C$ and the involved parameters are independent of $n \in \IN$, $s_1, \ldots, s_n$ and $t_1, \ldots, t_n$. Now, for fixed $r$ between $p$ and $q$ and therefore for fixed $\alpha \in (0,1)$ the angle $\theta$ is smaller than $\frac{\pi}{2}$ provided $\theta_1$ is chosen close enough to $\frac{\pi}{2}$. Consequently, for every $r$ between $p$ and $q$ there exists $\theta' \in (\varphi, \frac{\pi}{2})$ with
			\[
				\R[r]{ \lambda R(\lambda, A(t)): t \in [0,T], \lambda \in \partial \Sigma_{\theta'} \setminus \{0\}, \arg \lambda > 0} < \infty.
			\]
		We repeat the argument for $e^{-i\lambda}$. The assertion then follows from Lemma~\ref{lem:r_boundedness_analytic}.
	\end{proof}
	
	The above abstract criterion is particularly useful when the semigroups generated by $-A(t)$ are given by kernels satisfying uniform bounds. As a particular important instance we discuss Gaussian estimates in the next result.
	
	\begin{proposition}\label{prop:r_boundedness_kernel}
		Let $(\Omega, \Sigma, \mu)$ be a $\sigma$-finite measure space and $p \in (1, \infty)$. Suppose that $(A(t))_{t \in [0,T]}$ is a family of sectorial operators on $L^p(\Omega, \Sigma, \mu)$ such that the following holds. 
		\begin{thm_enum}
			\item\label{prop:r_boundedness_kernel:uniform_boundedness} There exists $\varphi \in (0,\frac{\pi}{2})$ such that $\sigma(A(t)) \subset \overline{\Sigma_{\varphi}}$ for all $t \in [0,T]$ and
				\[
					\sup_{t \in [0,T]} \sup_{\lambda \not\in \overline{\Sigma_{\varphi}}} \norm{\lambda R(\lambda, A(t))} < \infty.
				\]
			\item\label{prop:r_boundedness_kernel:kernel} The semigroups $(e^{-sA(t)})_{s \ge 0}$ satisfy uniform Gaussian estimates, i.e.\ for all $s > 0$ and $t \in [0,T]$ the operators $e^{-sA(t)}$ are given by kernels $k_{t}(x,y,s)$ for which there exist constants $C \ge 0$, $\beta > 0$ and $\omega_1 \in \IR$ such that for all $s > 0$ the estimate
				\[
					\abs{k_t(x,y,s)} \le C s^{-N/2} e^{\omega_1 s} \exp \biggl( -\frac{\abs{x-y}^2}{4\beta s} \biggr)
				\]
			holds almost everywhere on $\Omega \times \Omega$ and for all $t \in [0,T]$.
		\end{thm_enum} 
		Then there exists $\theta \in (\varphi, \frac{\pi}{2})$ such that for $\omega \ge \omega_1$
			\[
				\R{\lambda R(\lambda, A(t) + \omega): \lambda \not\in \overline{\Sigma_{\theta}}, t \in [0,T]} < \infty.
			\]
	\end{proposition}
	\begin{proof}
		By assumption~\ref{prop:r_boundedness_kernel:kernel} one has for all $t \in [0,T]$, $s > 0$ and $f \in L^p(\Omega) \cap L^2(\Omega)$, which after extension by zero we also regard as functions on $\IR^N$, the estimate
			\begin{align*}
				\normalabs{e^{-sA(t)}f}(x) & \le \int_{\Omega} \abs{k_{t}(x,y,s)} \abs{f(y)} \d y \\
				& \le C s^{-N/2} e^{\omega_1 s} \int_{\IR^N} \exp \biggl( -\frac{\abs{x-y}^2}{4\beta s} \biggr) \abs{f(y)} \d y \\
				& = C (4\pi)^{N/2} \beta^{N/2} e^{\omega_1 s} G_p(\beta s) \abs{f}(x), 
			\end{align*}
		where $(G_p(s))_{s \ge 0}$ is the heat semigroup on $L^p(\IR^N)$ given by
			\[
				G_p(s)f = k_s * f \qquad \text{with} \qquad k_s(x) = \frac{1}{(4\pi s)^{N/2}} e^{-\abs{x}^2/4s}.
			\]
		By density, the domination inequality is true for all $f \in L^p(\Omega)$. Writing the above estimate in a more compact form, we have for $\omega \ge \omega_1$ and a universal constant $M \ge 0$ the domination estimate
			\[
				\normalabs{e^{-\omega s} e^{-s A(t)} f}(x) \le M G_p(\beta s)\abs{f}(x) \eqqcolon (T(s)\abs{f})(x)
			\]
		for all $f \in L^p(\Omega)$, $s > 0$ and $t \in [0,T]$. We now show that the family $\mathcal{T} = \{ T(s): s \ge 0 \}$ is $\mathcal{R}_q$-bounded for all $q \in (1, \infty)$. First, as a consequence of the domination result for the maximal operator associated to convolutions~\cite[Theorem~2.1.10]{Gra08} one gets for all $s > 0$
			\[
				(T(s)\abs{f})(x) \le M \sup_{s > 0} (k_s * \abs{f})(x) \le \mathcal{M}(f)(x),
			\]
		where $\mathcal{M}$ is the centered Hardy--Littlewood maximal operator. Now, the $\mathcal{R}_q$-boundedness of $\mathcal{T}$ follows from domination and the fact that the Hardy--Littlewood maximal operator satisfies for all $q \in (1, \infty)$, $n \in \IN$ and $f_1, \ldots, f_n \in L^p(\IR^N)$ the vector-valued estimate~\cite[Theorem~4.6.6]{Gra08}
			\[
				\biggnorm{\biggl(\sum_{k=1}^n \abs{\mathcal{M}(f_k)}^q \biggr)^{1/q}}_p \lesssim \biggnorm{\biggl(\sum_{k=1}^n \abs{f_k}^q \biggr)^{1/q}}_p.
			\]
		 Hence, the result follows from Proposition~\ref{prop:r_boundedness} applied to the family of sectorial operators $(A(t) + \omega)_{t \in [0,T]}$.
	\end{proof}
	
	As a consequence of this result and the heat kernel estimates shown by Daners in~\cite{Dan00} we obtain the following $\mathcal{R}$-boundedness results dealing with Dirichlet, Neumann and Robin boundary conditions respectively.
	
	\begin{theorem}[Dirichlet boundary conditions]\label{thm:r_boundedness_dirichlet}
		Let $\Omega$ be a domain in $\IR^N$. Suppose that one has real-valued coefficients $a_{ij}, a_i, b_i, c_0 \in L^{\infty}([0,T] \times \Omega)$ with
			\[
				\sum_{i,j = 1}^N a_{ij}(t,x) \xi_i \xi_j \ge \alpha_0 \abs{\xi}^2
			\]
		for some $\alpha_0 > 0$. For $p \in (1, \infty)$ and $t \in [0,T]$ let $A(t)$ be the $L^p$-realization of the operator associated to the form~\eqref{eq:form_dn} with form domain $W_0^{1,2}(\Omega)$. Then there exist $\omega_0 \ge 0$ and $\theta \in (0, \frac{\pi}{2})$ such that for all $\omega \ge \omega_0$ one has $\sigma(A(t) + \omega) \subset \overline{\Sigma_{\theta}}$ for all $t \in [0,T]$ and
			\[
				\R{\lambda R(\lambda, A(t) + \omega): \lambda \not\in \overline{\Sigma_{\theta}}, t \in [0,T]} < \infty.
			\]
		Further, in the case $\norm{a_i}_{\infty} = \norm{b_i}_{\infty} = \norm{\min (c_0, 0)}_{\infty} = 0$ we can choose $\omega_0 = 0$.
	\end{theorem}
	\begin{proof}
		It follows from~\cite[Theorem~6.1, discussion after Theorem~5.1]{Dan00} that the semigroups $(e^{-sA(t)})_{s \ge 0}$ are given by kernels $k_t(x,y,s)$ satisfying
			\[
				\abs{k_t(x,y,s)} \le C_N (2\alpha_0)^{-N/2} C^{N/2} s^{-N/2} e^{2\omega_1s} \exp\biggl( - \frac{\abs{x-y}^2}{8 \omega_2 s} \biggr).
			\]
		Whereas $C_N$ only depends on $N$ and $C$ is the constant of some Sobolev inequality, it is shown that $\omega_1$ depends only on upper bounds for $\alpha_0^{-1}$, $\norm{a_i(t, \cdot)}_{\infty}$, $\norm{b_i(t, \cdot)}_{\infty}$ for $i = 1, \ldots, N$ and $\norm{\min (c_0(t, \cdot), 0)}_{\infty}$. Further, $\omega_2$ depends on upper bounds for $\norm{a_{ij}(t, \cdot)}_{\infty}$ and $\alpha_0^{-1}$. Hence, the assumed uniform ellipticity and the uniform boundedness of the coefficients in the time variable imply that assumption~\ref{prop:r_boundedness_kernel:kernel} of Theorem~\ref{prop:r_boundedness_kernel} is satisfied for all $\omega \ge 2 \omega_1$. Moreover, $\omega_1 = 0$ in the case of the addendum. It remains to remark that also assumption~\ref{prop:r_boundedness_kernel:uniform_boundedness} holds: for this recall from Section~\ref{sec:form_method} that the analyticity properties on $L^2$ only depend on the norm and on the coercivity constant of the associated forms $a(t,\cdot, \cdot) + \tilde{\omega} (\cdot, \cdot)_H$ if $\tilde{\omega}$ can be chosen independently of $t \in [0,T]$. Moreover, these constants and the choice of $\tilde{\omega}$ only depend on upper bounds for $\alpha_0^{-1}$ and the essential supremas of the coefficient functions. Hence, assumption~\ref{prop:r_boundedness_kernel:uniform_boundedness} is satisfied on $L^2$. Further, we know from the kernel bounds that the operators $-A(t)$ generate $C_0$-semigroups with uniform growth bounds. This shows assumption~\ref{prop:r_boundedness_kernel:uniform_boundedness} for $p \in (1, \infty)$, but $\theta > \frac{\pi}{2}$. In order to deduce uniform sectorial estimates for some $\theta < \frac{\pi}{2}$, i.e.\ assumption~\ref{prop:r_boundedness_kernel:uniform_boundedness}, one now invokes the classical Stein interpolation theorem (\cite[Theorem~1.3.7]{Gra08}, for an alternative approach to this result see~\cite[Theorem~3.1]{Fac13b}) by interpolating between the $L^2$ and the $L^p$ result in the spirit of the proof of Proposition~\ref{prop:r_boundedness}.
	\end{proof}
	
	For the two remaining boundary conditions we can invoke~\cite[Theorem~6.1]{Dan00} in the same manner. We therefore only state the results.
	
	\begin{theorem}[Neumann boundary conditions]\label{thm:r_boundedness_neumann}
		Let $\Omega$ be a domain in $\IR^N$ satisfying the interior cone condition. Suppose that one has real-valued coefficients $a_{ij}, a_i, b_i, c_0 \in L^{\infty}([0,T] \times \Omega)$ with
			\[
				\sum_{i,j = 1}^N a_{ij}(t,x) \xi_i \xi_j \ge \alpha_0 \abs{\xi}^2
			\]
		for some $\alpha_0 > 0$. For $p \in (1, \infty)$ and $t \in [0,T]$ let $A(t)$ be the $L^p$-realization of the operator associated to the form~\eqref{eq:form_dn} with form domain $W^{1,2}(\Omega)$. Then there exist $\omega_0 \ge 0$ and $\theta \in (0, \frac{\pi}{2})$ such that for all $\omega \ge \omega_0$ one has $\sigma(A(t) + \omega) \subset \overline{\Sigma_{\theta}}$ for all $t \in [0,T]$ and
			\[
				\R{\lambda R(\lambda, A(t) + \omega): \lambda \not\in \overline{\Sigma_{\theta}}, t \in [0,T]} < \infty.
			\]
	\end{theorem}
	
	\begin{theorem}[Robin boundary conditions]\label{thm:r_boundedness_robin}
		Let $\Omega$ be a bounded Lipschitz domain in $\IR^N$. Suppose that one has real-valued coefficients $a_{ij}, a_i, b_i, c_0 \in L^{\infty}([0,T] \times \Omega)$ and $\beta \in L^{\infty}([0,T] \times \partial \Omega)$ with $\beta \ge 0$ and
			\[
				\sum_{i,j = 1}^N a_{ij}(t,x) \xi_i \xi_j \ge \alpha_0 \abs{\xi}^2
			\]
		for some $\alpha_0 > 0$. For $p \in (1, \infty)$ and $t \in [0,T]$ let $A(t)$ be the $L^p$-realization of the operator associated to the form~\eqref{eq:form_robin} with form domain $W^{1,2}(\Omega)$. Then there exist $\omega_0 \ge 0$ and $\theta \in (0, \frac{\pi}{2})$ such that for all $\omega \ge \omega_0$ one has $\sigma(A(t) + \omega) \subset \overline{\Sigma_{\theta}}$ for all $t \in [0,T]$ and
			\[
				\R{\lambda R(\lambda, A(t) + \omega): \lambda \not\in \overline{\Sigma_{\theta}}, t \in [0,T]} < \infty.
			\]
	\end{theorem}
	
	\begin{remark}
		The above Gaussian bounds can be generalized to complex measurable coefficients provided the leading coefficients $(a_{ij})$ satisfy some additional regularity assumptions~\cite[Section~5]{Ouh04}. Consequently, for such coefficients an analogue of the above $\mathcal{R}$-boundedness results holds. Furthermore, if the domain is assumed to be $C^{1}$ and the leading coefficients lie in $VMO(\Omega)$, then without discussing the dependencies of the constant Gaussian estimates were obtained in~\cite[Theorem~7]{AusTch01b} for complex coefficients.
	\end{remark}
	
	\subsection{Uniformly bounded imaginary powers}
	
	We now prove uniform bounds on the imaginary powers of second order operators. Note that it is known~\cite[Theorem~3.1]{DuoRob96} that Gaussian estimates imply the boundedness of the $H^{\infty}$-calculus (we refer to~\cite{Haa06}) and therefore in particular the boundedness of the imaginary powers. However, verifying uniform constants would involve yet another constant chasing. Therefore, we take another route via dilation theory and transference methods.
	
	\begin{theorem}\label{thm:bip_second_order}
		Consider operators as in~Theorem~\ref{thm:r_boundedness_dirichlet}, \ref{thm:r_boundedness_neumann}, or~\ref{thm:r_boundedness_robin}. Then for all $p \in (1, \infty)$ there exists $\omega_0 \ge 0$ such the associated operators $L_p(t) + \omega$ have uniformly bounded imaginary powers for all $\omega \ge \omega_0$. If $c_0$ and in the Robin case $\beta$ are non-negative and all other coefficients except for $(a_{ij})$ vanish, then one can choose $\omega_0 = 0$.
	\end{theorem}
	\begin{proof}
		Fix $p \in (1, \infty)$. It is shown in~\cite[Theorem~4.3 \& Remark~4.4]{Nit12} that there exists $\omega_p \ge 0$ such that the semigroups generated by $-L_p(t)$ are positive and satisfy
		\[
			\normalnorm{e^{-sL_p(t)}}_{p} \le e^{\omega_p s}.
		\]
		In fact, the statement there contains an explicit upper bound for $\omega_p$ which only depends on upper bounds for $\alpha_0^{-1}$ in~\eqref{eq:ellipticity_condition}, the $L^{\infty}$-norms of the coefficients and of $\beta$ in the case of Robin boundary conditions. In particular, $(e^{-\omega_p s} T_t(s))_{t \ge 0}$ is a semigroup of positive contractions for all $t \in [0,T]$. By Fendler's dilation theorem~\cite{Fen97} and the transference principle of Coifman--Weiss~\cite{CoiWei76} (used after rewriting the estimate into an estimate for the Hille--Phillips calculus with the help of~\cite[Lemma~2.12]{Mer99}) which both do not increase the size constants of the imaginary powers (see also~\cite{HiePru98} where this strategy was generalized), the problem is reduced to the boundedness of the imaginary powers of the shift group on $L^p(\IR)$ alone. Now, it is well-known (see for example~\cite[Theorem~3.1]{DorVen87} for the analogous case of the unilateral shift on $L^p([0,\infty))$) that if $-A$ generates the shift group, then for all $\eta > \frac{\pi}{2}$ there exists $M > 0$ with
			\[
				\normalnorm{A^{is}} \le M e^{\eta \abs{s}} \qquad \text{for all } s \in \IR.
			\]  
		Hence, $L_p(t) + \omega$ has uniformly bounded imaginary powers for all $\omega \ge \omega_p$. 
	\end{proof}

\section{Maximal Regularity for Non-Autonomous Boundary Conditions}
	
	In this section we apply our results to the Laplacian with non-autonomous boundary conditions on sufficiently smooth domains. This can be considered as the easiest model case for non-autonomous boundary conditions. The domains of the involved operators will be time dependent. In fact, due to the easy structure of the example the domains of the corresponding operators are explicitly known. In the next section we will consider time dependent elliptic operators in divergence form with rough coefficients for which the domains cannot be determined. Although the situation looks quite simple, this example gives a good illustration of the assumptions and limitations of Theorem~\ref{thm:abstract_mr_non-zero} and the result seems to be new.
		 
	 We now fix the setting. Let $\Omega \subset \IR^N$ be a bounded domain with $C^2$ boundary and $\beta\colon [0,T] \times \partial \Omega \to \IR_{\ge 0}$ be a non-negative function with the following properties.
	
	\begin{thm_enum}
		\item The functions $\beta(t,\cdot)$ are uniformly Lipschitz in $t \in [0,T]$, i.e.\ there exists $L > 0$ with
			\[
				\abs{\beta(t,x) - \beta(t,y)} \le L \abs{x-y} \qquad \text{for all } x, y \in \partial \Omega, \, t \in [0,T].
			\]
		\item The functions $\beta(\cdot,x)$ are uniformly Hölder continuous for some $\alpha > 0$, i.e.\ there exists $M \ge 0$ with
			\[
				\abs{\beta(t,x) - \beta(s,x)} \le M \abs{t-s}^{\alpha} \qquad \text{for all } t,s \in [0,T], \, x \in \partial \Omega. 
			\]
	\end{thm_enum}
	
	The second assumption is needed for Theorem~\ref{thm:abstract_mr_non-zero}, whereas the first is in some sense out of convenience. Indeed, it allows us to identify the domain of the Robin Laplacian in $L^p(\Omega)$. Let us fix some $p \in (1, \infty)$. Further, for $t \in [0,T]$ let $A(t)$ be the realization of the Robin problem on $L^p(\Omega)$ induced by the form
	\begin{equation*}
		W^{1,2}(\Omega) \times W^{1,2}(\Omega) \ni (u,v) \mapsto \int_{\Omega} \nabla u \cdot \nabla v + \int_{\Omega} u v + \int_{\partial \Omega} \beta(t, \cdot) u v \d\mathcal{H}_{N-1}.
	\end{equation*}
	Due to the assumed regularity on the boundary and on $\beta$ it follows from~\cite[Theorem~2.4.2.7]{Gri85} that
	\begin{align*}
		D(A(t)) & = \{ u \in W^{2,p}(\Omega): \partial_v u + \beta(t, \cdot) u = 0 \} \\
		A(t)u & = -\Delta u + u,
	\end{align*}	
	where $\partial_{\nu}$ denotes the outer normal derivative. Let us first remark that assumption~\ref{thm:abstract_mr_non-zero:r_boundedness} of Theorem~\ref{thm:abstract_mr_non-zero} is fulfilled after some suitable shift as a particular instance of Theorem~\ref{thm:r_boundedness_robin}. Hence, it remains to check the Hölder assumption~\ref{thm:abstract_mr_zero:hoelder} of Theorem~\ref{thm:abstract_mr_non-zero}. We give two approaches that illustrate the interplay of the assumptions in the theorem.
	
	\subsection{First approach}
	
	It is shown in~\cite{Sob92} that if $\beta(t,\cdot)$ is additionally assumed to be smooth, then the domains $D(A(t)^{1/2})$ are independent of $t$ and for all $t,s \in [0,T]$ and $p \in (1, \infty)$ the difference $A^{1/2}(t) - A^{1/2}(s)$ extends to a bounded operator on $L^p(\Omega)$ satisfying
		\begin{equation*}
			\normalnorm{A^{1/2}(t) - A^{1/2}(s)}_{\mathcal{B}(L^p(\Omega))} \le K \norm{\beta(t,\cdot) - \beta(s,\cdot)}_{\infty} \le KM \abs{t-s}^{\alpha},
		\end{equation*}
	where $K$ is a constant only depending on $p$ and $\Omega$. Now, note that for $u \in D(A(s))$ and $v \in D(A(t)')$ one obtains the time regularity estimate
		\begin{align*}
			\MoveEqLeft \abs{\langle (A_{-1}(t) - A_{-1}(s))u, v \rangle} = \normalabs{\langle A(t)^{1/2}u, A(t)'^{1/2}v \rangle - \langle A(s)^{1/2}u, A(s)'^{1/2}v \rangle} \\
			& = \normalabs{\langle (A(t)^{1/2} - A(s)^{1/2})u, {A(t)'}^{1/2}v \rangle + \langle (A(s)^{1/2}u, ({A(t)'}^{1/2}  - {A(s)'}^{1/2}) v \rangle} \\
			& \le \normalnorm{(A(t)^{1/2} - A(s)^{1/2})u}_{L^p(\Omega)} \normalnorm{{A(t)'}^{1/2}v}_{L^q(\Omega)} \\
			& \qquad + \normalnorm{A(s)^{1/2} u}_{L^p(\Omega)} \normalnorm{({A(t)'}^{1/2} - {A(s)'}^{1/2})v}_{L^q(\Omega)},
		\end{align*}
	where $\frac{1}{p} + \frac{1}{q} = 1$. Putting the two inequalities together, we therefore obtain
		\begin{align*}
			\MoveEqLeft \abs{\langle (A_{-1}(t) - A_{-1}(s))u, v' \rangle} \lesssim \abs{t-s}^{\alpha} \norm{u}_{D(A(s)^{1/2})} \norm{v}_{D(A(t)'^{1/2})}.
		\end{align*}
	It follows from~Theorem~\ref{thm:bip_second_order} together with Proposition~\ref{prop:bip_fractional} that we can exchange $D(A(t)^{1/2})$ and $D(A(s)'^{1/2})$ with complex interpolation spaces of order $\frac{1}{2}$ at the cost of a universal constant. Now, Theorem~\ref{thm:abstract_mr_non-zero} applied to $(A(t) + \omega)_{t \in [0,T]}$ for the complex interpolation scale and some sufficiently large $\omega \ge 0$ implies maximal regularity provided $\alpha > 1/2$. However, one can do better and lower the required regularity on $\beta$. 
	
	\subsection{Second approach}
	
	On an intuitive level this improvement stems from the fact that the difference $A_{-1}(t) - A_{-1}(s)$ only involves the boundary term which is of lower order than the main terms. Indeed, for $p \in (1, \infty)$ one has $\Tr\colon W^{r,p}(\Omega) \to L^p(\Omega)$ for $r > \frac{1}{p}$ \cite[Theorem~7.43 \& Remark~7.45~2.]{AdaFou03} under our assumptions made on $\Omega$. Thus for $r > \frac{1}{q}$, where $\frac{1}{p} + \frac{1}{q} = 1$, and $u \in D(A(s))$ and $v \in D(A(t)')$ we obtain
		\begin{align*}
			\MoveEqLeft \abs{\langle A_{-1}(t) - A_{-1}(s))u, v \rangle} = \abs{\int_{\partial \Omega} (\beta(t,\cdot) - (\beta(s,\cdot)) u_{| \partial \Omega} v_{| \partial \Omega} \d\mathcal{H}_{N-1}} \\
			& \le \norm{\beta(t,\cdot) - \beta(s,\cdot)}_{\infty} \normalnorm{u_{|\partial \Omega}}_{L^p(\partial \Omega)} \normalnorm{v_{|\partial \Omega}}_{L^q(\partial \Omega)} \\
			& \lesssim \abs{t-s}^{\alpha} \norm{u}_{W^{2,p}(\Omega)} \norm{v}_{W^{r,q}(\Omega)}.
		\end{align*}
		Here the first identity follows from approximating elements in $D(A(s))$ and $D(A(t)')$ by elements that additionally lie in the respective $L^2$-realizations. We now need an estimate for the domains of the operators.
		
	\begin{lemma}
		Let $\Omega \subset \IR^N$ be a bounded domain with $C^2$ boundary, $p \in (1, \infty)$ and $I$ an index set. Further let $\beta_i\colon \partial \Omega \to \IR_{\ge 0}$ for $i \in I$ be a uniformly bounded family in $W^{1,\infty}(\partial \Omega)$ and denote by $A_i$ the realization of the Robin Laplacian on $L^p(\Omega)$. Then there exists a constant $C \ge 0$ such that for all $i \in I$ and all $u \in D(A_i)$ one has
			\[
				\norm{u}_{W^{2,p}(\Omega)} \le C (\norm{A_i u}_p + \norm{u}_p).
			\]
	\end{lemma}
	\begin{proof}
		It follows from Theorem~\cite[Theorem~2.3.3.6]{Gri85} applied to Neumann boundary conditions that there exist constants $C \ge 0$ and $\lambda > 0$ with
			\[
				\norm{u}_{W^{2,p}(\Omega)} \le C (\norm{\Delta u + \lambda u}_{L^p(\Omega)} + \norm{\partial_{\nu} u}_{W^{1-\frac{1}{p},p}(\partial \Omega)})
			\]
		for all $u \in W^{2,p}(\Omega)$. Now, if $u$ additionally lies in $D(A_i)$, then it follows that
			\begin{align*}
				\norm{u}_{W^{2,p}(\Omega)} & \lesssim \norm{\Delta u + u}_p + \norm{u}_p + \norm{\beta_i u}_{W^{1-\frac{1}{p},p}(\partial \Omega)} \\
				& \lesssim \norm{A_i u}_p + \norm{u}_p + \norm{\beta_i}_{W^{1,\infty}(\partial \Omega)} \norm{u}_{W^{1,p}(\partial \Omega)}.
			\end{align*}
		For every $\epsilon > 0$ there exists $C_{\epsilon} > 0$ such that $\norm{u}_{W^{1,p}(\Omega)} \le \epsilon \norm{u}_{W^{2,p}(\Omega)} + C_{\epsilon} \norm{u}_{L^p(\Omega)}$ as a consequence of the compactness of the embedding $W^{2,p}(\Omega) \hookrightarrow W^{1,p}(\Omega)$. This inequality for some sufficiently small $\epsilon$ together with the uniform boundedness of $\beta_i$ in $W^{1,\infty}(\Omega)$ yields the assertion.
	\end{proof}

		Now, interpolating the uniform embedding $D(A(t)') \hookrightarrow W^{2,q}(\Omega)$ obtained in the lemma against the identity $L^q(\Omega) \to L^q(\Omega)$, we obtain the uniform estimate 
		\[
			\norm{u}_{W^{r,q}(\Omega)} \lesssim \norm{u}_{(L^q(\Omega), D(A(t)'))_{r/2, q}}.
		\]
		Applying Theorem~\ref{thm:abstract_mr_non-zero} for the scale given by $(\cdot, \cdot)_{\theta,p}$, we improve the Hölder regularity in the time variable to $\beta > \frac{1}{2q} = \frac{1}{2} - \frac{1}{2p}$. We have shown the following.

	\begin{theorem}
		Let $\Omega \subset \IR^N$ be a bounded domain with $C^2$ boundary and $p \in (1, \infty)$. Suppose that $\beta\colon [0,T] \times \partial \Omega \to \IR_{\ge 0}$ satisfies for some $\alpha > \frac{1}{2} - \frac{1}{2p}$ and $M, L \ge 0$ the assumptions
		\begin{thm_enum}
    		\item $\abs{\beta(t,x) - \beta(t,y)} \le L \abs{x-y}$ for all $x, y \in \partial \Omega$ and $t \in [0,T]$,
    		\item $\abs{\beta(t,x) - \beta(s,x)} \le M \abs{t-s}^{\alpha}$ for all $t,s \in [0,T]$ and $x \in \partial \Omega$.
		\end{thm_enum}
		Then the non-autonomous Robin problem
    		\begin{equation*}
    			\left\{
    			\begin{aligned}
    				u_t(t,x) - \Delta u(t,x) & = f(t,x) & \text{on } (0,T) \times \Omega \\
					\partial_{\nu} u(t,x) + \beta(t,x) u(t,x) & = 0 & \text{on } (0,T) \times \partial\Omega \\
    				u(0,x) & = u_0(x) & \text{on } \Omega
    			\end{aligned}
    			\right.
    		\end{equation*}
		has maximal $L^q(L^p)$-regularity for all $q \in (1, \infty)$, i.e.\ for all $f \in L^q([0,T];L^p(\Omega))$ and all  initial values $u_0 \in D(A(0),X)_{1/q,q}$ the problem has a unique solution $u \in W^{1,q}([0,T];L^p(\Omega))$ with $u(t) \in D(A(t))$ for almost all $t \in [0,T]$ and $\Delta u \in L^q([0,T];L^p(\Omega))$. Furthermore, there exists a constant $C_{p,q} > 0$ independent of $f$ and $u_0$ with
			\begin{align*}
				\MoveEqLeft \norm{u}_{W^{1,q}([0,T];L^p(\Omega))} + \norm{u}_{L^q([0,T];W^{2,p}(\Omega))} \\
				& \le C_{p,q} (\norm{u_0}_{D(A(0),X)_{1/q,q}} + \norm{f}_{L^q([0,T];L^q(\Omega))}).
			\end{align*}
	\end{theorem}
	
	\begin{remark}\label{rem:weights}
		Note that if $\beta > \max \{\frac{1}{2} - \frac{1}{2p}, \frac{1}{2p} \}$, then Theorem~\ref{thm:abstract_mr_non-zero} shows that both $(A(t))_{t \in [0,T]}$ and $(A(t)')_{t \in [0,T]}$ satisfy the Acquistapace--Terreni condition. Hence, \cite[Theorem~5.3]{ChiKro14} shows that for $u_0 = 0$ maximal $L^q$-regularity extrapolates to all weights $w \in A_q^{-}(\IR_{>0})$, e.g.\ $w(t) = t^{\alpha}$ for $\alpha \in (-\infty, q-1)$.  
	\end{remark}
	
	\begin{remark}
		As already said, the Lipschitz assumption in the spatial variables is made for the ease of presentation because it allows us to identify the domain of the operator explicitly. Note that for the last proof to work the validity of the weaker estimates $\norm{u}_{W^{1,p}(\Omega)} \lesssim \norm{A(t)u}_p + \norm{u}_p$ and $\norm{v}_{W^{1,q}(\Omega)} \lesssim \norm{v}_{D(A(s)'^{1/2})}$, i.e.\ the boundedness of the associated Riesz transform on $L^q(\Omega)$, would be sufficient. The first estimate holds in particular if $\Omega$ has $C^{1}$ boundary and $\beta$ is merely assumed to lie in $L^{\infty}([0,T] \times \partial \Omega)$~\cite[Theorem~5(ii)]{DonKim10}, whereas for the second we do not know any explicit reference. For details we refer to the next section where the same approach is carried out for second order elliptic operators in divergence form.
	\end{remark}
	
\section{Maximal Regularity for Non-Autonomous Second Order Elliptic Operators in Divergence Form}

	In our second example we apply Theorem~\ref{thm:abstract_mr_non-zero} to time dependent second order elliptic operators in divergence form with real coefficients. The operators are studied on $C^1$ domains subject to Dirichlet boundary conditions. As the used results are rather involved, we neither discuss the case of complex coefficients nor lower order terms and different boundary conditions. However, extensions in these directions are certainly possible as long as the used regularity results are true. Having this in mind, we will comment on the validity of the used results in more general situations.
	
	Let us begin with recalling some concepts from the theory of elliptic equations. Let $\Omega \subset \IR^N$ be a bounded domain and assume that $A = (a_{ij})_{1 \le i,j \le N}\colon \Omega \to \IR^{N \times N}$ is bounded and measurable and satisfies the ellipticity condition~\eqref{eq:ellipticity_condition}. For given $f \in L^p(\Omega)$ we say that the function $u \in W^{1,p}_0(\Omega)$ is a \emph{weak solution} of the inhomogeneous elliptic problem
	\begin{equation*}
		\tag{EP}
		\label{eq:elliptic_problem}
		-\divv(A \nabla u) = f
	\end{equation*}
	subject to Dirichlet boundary conditions if one has
	\begin{equation*}
		\int_{\Omega} A \nabla u \cdot \nabla \phi = \int_{\Omega} f \phi \qquad \text{for all } \phi \in C_c^{\infty}(\Omega).
	\end{equation*}
	The above equality extends to all $\phi \in W_0^{1,q}(\Omega)$. Consider now the form
	\begin{align*}
		W_0^{1,2}(\Omega) \times W_0^{1,2}(\Omega) \ni (u,v) & \mapsto \int_{\Omega} A \nabla u \cdot \nabla v,
	\end{align*} 
	a special case of~\eqref{eq:form_dn}. Let $L_2$ be the associated operator on $L^2(\Omega)$ as explained in Section~\ref{sec:form_method}. Then it follows from the definition of $L_2$ that for $u \in W_0^{1,2}(\Omega)$ one has $u \in D(L_2)$ with $L_2 u = f$ if and only if
		\begin{equation*}
			\int_{\Omega} A \nabla u \cdot \nabla \phi = \int_{\Omega} f \phi \qquad \text{for all } \phi \in C_c^{\infty}(\Omega).
		\end{equation*}
	As shown in Section~\ref{sec:form_elliptic_operator} one can extrapolate the operator $L_2$ to $L^p(\Omega)$ for $p \in (1, \infty)$. We denote by $L_p$ the extrapolated operator on $L^p(\Omega)$. Note that if $p > 2$, then $L_p$ is the part of $L_2$ in $L^p(\Omega)$, i.e.\
		\[
			D(L_p) = \{ u \in L^p(\Omega): L_2 u \in L^p(\Omega) \}.
		\]
	If the coefficients $(a_{ij})$ are merely continuous or measurable, then one cannot determine the domain of $L_p$ explicitly. However, one has $D(L_2) \subset W_0^{1,2}(\Omega)$.
	
\subsection{The higher dimensional case}
	
	Under weak regularity assumptions on the coefficients the inclusion $D(L_2) \subset W_0^{1,2}(\Omega)$ remains true for $p \in (1, \infty)$ and sufficiently regular domains $\Omega \subset \IR^N$. In order to state the results we introduce some new function spaces. We say that a locally integrable function $f\colon \Omega \to \IR$ lies in $BMO(\Omega)$ (put in words, $f$ has \emph{bounded mean oscillation}) if
	\begin{equation*}
		\sup_{B} \frac{1}{\abs{B}} \int_{B} \abs{f(x) - f_B} \d x < \infty,
	\end{equation*}
	where the supremum is taken over all non-empty balls $B \subset \Omega$ of finite measure and $f_B$ denotes the mean of $f$ over $B$. Further, for $f \in BMO(\Omega)$ and $r > 0$ let
	\begin{equation*}
		\eta(r) = \sup_{\rho \le r} \frac{1}{\abs{B_{\rho}}} \int_{B_{\rho}} \abs{f(x) - f_B} \d x.
	\end{equation*}
	Here the supremum runs over all balls $B_{\rho} \subset \Omega$ with radius $\rho$ smaller than $r$. We say that $f \in VMO(\Omega)$ (for \emph{vanishing mean oscillation}) if $\lim_{r \downarrow 0} \eta(r) = 0$ and call $\eta$ the \emph{vmo-modulus} of $f$.
	
	For our further purposes it is important to get uniform control over the involved constants. In the following we will say that the constant only depends on the vmo-modulus of continuity. This means the following: suppose that for a family of functions with vmo-moduli $\eta_i$ there exists a function $\eta\colon (0,1] \to \overline{\IR_{\ge 0}}$ with $\lim_{r \downarrow 0} \eta(r) = 0$ and $\eta_i(r) \le \eta(r)$ for all $r \in (0,1]$ and $i \in I$. Then the constant can be chosen independently of the functions.
	
	For coefficients with vanishing mean oscillation we obtain the inclusion $D(L_p) \subset W_0^{1,p}(\Omega)$ for domains with $C^1$ boundary.
		
	\begin{proposition}\label{prop:sobolev_elliptic_operator}
		For $N \ge 2$ let $\Omega \subset \IR^N$ be a bounded domain with $C^1$ boundary and $p \in [2, \infty)$. Further choose $A \in L^{\infty}(\Omega; \IR^{N\times N})$ that satisfies the ellipticity estimate \eqref{eq:ellipticity_condition} and whose components are in $VMO(\Omega)$. Then $D(L_p) \subset W_0^{1,p}(\Omega)$ and there exists a constant $C \ge 0$ only  depending on $\Omega$, $p$, $N$, $\alpha_0$ in \eqref{eq:ellipticity_condition}, upper bounds for $\norm{A}_{\infty}$ and the vmo-moduli of continuity of the entries of $A$ such that
		\[
			\norm{u}_{W_0^{1,p}(\Omega)} \le C \norm{u}_{D(L_p)} \qquad \forall u \in D(L_p).
		\] 
	\end{proposition}
	\begin{proof}
		Let $u \in D(L_p)$ and set $f = L_p u \in L^p(\Omega)$. Since a fortiori $u$ is contained in $D(L_2)$, we have
			\begin{equation*}
				\label{eq:weak_solution_elliptic}
				\int_{\Omega} A \nabla u \cdot \nabla \phi = \int_{\Omega} f \phi \qquad \text{for all } \phi \in C_c^{\infty}(\Omega).
			\end{equation*}
		Hence, $u$ is a weak solution of the problem~\eqref{eq:elliptic_problem}. Now, note that the Dirichlet Laplacian on $\Omega$ is an isomorphism $\Delta\colon D(\Delta) \xrightarrow{\sim} L^p(\Omega)$. Hence, there exists a potential $\Phi \in D(\Delta)$ with $\Delta \Phi = f$. In particular by~\cite[Theorem~B]{JerKen95}, one has $\nabla \Phi \in L^p(\Omega; \IR^N)$ with $\norm{\nabla \Phi}_p \le C \norm{f}_p$ for some $C \ge 0$ independent of $f$. It follows from~\cite[Theorem~1]{AusQaf02} that the problem
			\[
				\divv(A\nabla v) = \divv(\nabla \Phi) = f
			\]
		has a unique solution $v \in W_0^{1,p}(\Omega)$ with $\norm{\nabla v}_{p} \lesssim \norm{\nabla \Phi}_p$. Furthermore, the implicit constant only depends on on $\Omega$, $p$, $N$, $\alpha_0$ in \eqref{eq:ellipticity_condition}, upper bounds for $\norm{A}_{\infty}$ and the vmo-moduli of continuity of the entries of $A$. For the dependence on the vmo-moduli observe that the moduli only start to play a role from~\cite[page~506]{AusQaf02} on, where they are used to make the mean oscillation of some localized coefficients arbitrarily small. Now,
			\begin{align*}
				\norm{\nabla v}_p \lesssim \norm{\nabla \Phi}_p \lesssim \norm{f}_p = \norm{L_p u}_p.
			\end{align*}
		Comparing the solution $v$ with~\eqref{eq:weak_solution_elliptic}, we see that $v = u$ by the uniqueness of the weak solution. Hence, $u \in W_0^{1,p}(\Omega)$ and the desired estimate holds.
	\end{proof}
	
	\begin{remark}\label{rem:more_general_estimate}
		In fact, the result~\cite[Theorem~1]{AusQaf02} also includes complex coefficients under the same assumptions. Even more, the Dirichlet boundary conditions can be replaced by Neumann boundary conditions. Of course, in this case  the solution is only unique up to constants. For fixed $p \in (1, \infty)$ one can extend the above result to coefficients in $BMO$ and to Lipschitz domains, provided the involved constants are small enough~\cite{Byu05}. Further, the result holds if loosely spoken the coefficients are only measurable in one component and $VMO$ in the other components~\cite[Theorem~4]{DonKim10}.		
	\end{remark}

	Even for smooth domains Proposition~\ref{prop:sobolev_elliptic_operator} does not hold for arbitrary bounded measurable coefficients as the following slight adaption of an example due to Meyers \cite[Section~5]{Mey63} (see also~\cite[p.~1026]{Byu05}) shows.
	
	\begin{example}\label{ex:counterexample_regularity}
		On the unit ball $B \subset \IR^2$ consider the real symmetric coefficient matrix
		\[
			A(x,y) = \frac{1}{4(x^2+y^2)} \begin{pmatrix}
					4x^2 + y^2 & 4xy \\
					4xy & x^2+4y^2
				\end{pmatrix}.
		\]
		One can verify that $u(x,y) = \frac{x}{(x^2+y^2)^{1/4}}$ is a solution of the problem $\divv(A \nabla u) = 0$ for which $\nabla u \not\in L^p(B)$ holds for $p \ge 4$. Now, choose a smooth cut-off function $0 \le \psi \le 1$ centered at the origin with compact support in $B$. Consider the function $w = u \psi$. Then for all points different from zero we have
		\begin{align*}
			\MoveEqLeft \divv(A\nabla w) = \divv(A \nabla (u \psi)) = \divv(\psi A \nabla u + u A \nabla \psi) \\
			& = \psi \divv(A \nabla u) + \nabla \psi \cdot A \nabla u + u \divv(A \nabla \psi) + A \nabla \psi \cdot \nabla u \\
			& = 2 A \nabla \psi \cdot \nabla u + u \divv(A \nabla \psi).
		\end{align*}
		Clearly, this function is $C^{\infty}$ away from zero. Further, in a neighborhood of zero one has $\nabla \psi = 0$ and therefore the right hand side vanishes there. Since $w$ is compactly supported, this shows that $Lw \in L^p(\Omega)$ for all $p \in [1, \infty]$. As before denote by $L_p$ the realization of $-\divv(A \nabla \cdot)$. Since $D(L_p)$ is the part of $L_2$ in $L^p(B)$ for $p > 2$, we get $w \in D(L_p)$ for all $p \in (2, \infty)$. However, since $w = u$ in neighborhood of zero, we have $\nabla w \not\in L^p(B)$ for $p \ge 4$. Hence, $D(L_p)$ contains elements that are not in $W^{1,p}(B)$.
	\end{example}
	
	Further, even for the Laplacian some regularity on the domain is necessary to obtain the assertion of Proposition~\ref{prop:sobolev_elliptic_operator}. In fact,  it is known~\cite[Theorem~A]{JerKen95} that for $p > 3$ there exist Lipschitz domains for which the domain of the Dirichlet Laplacian on $L^p(\Omega)$ contains elements that are not in $W^{1,p}(\Omega)$.  
	
	The result given in Proposition~\ref{prop:sobolev_elliptic_operator} is also true for $p \in (1,2)$, but in this case even more holds. Recall that $L_p$ is sectorial. Thus its fractional powers are well-defined.
		
	\begin{theorem}\label{thm:kato_square_root}
		For $N \ge 2$ let $\Omega \subset \IR^N$ be a bounded Lipschitz domain and $p \in (1, 2]$. Further choose $A \in L^{\infty}(\Omega; \IR^{N \times N})$ that satisfies the ellipticity estimate~\eqref{eq:ellipticity_condition}. Then there exists a constant $C \ge 0$ only depending on $N$, $p$, $\alpha_0$ in~\eqref{eq:ellipticity_condition}, the Lipschitz constant of $\Omega$ and on upper bounds for $\norm{A}_{\infty}$ such that
		\[
			\norm{\nabla f}_p \le C \normalnorm{L_p^{1/2} f}_p \qquad \text{for all } f \in D(L_p^{1/2}). 
		\]
	\end{theorem}
	\begin{proof}
		It follows from the solution of the Kato square root problem on Lipschitz domains~(\cite[Theorem~1]{AKM06}, see also the first proof for the Dirichlet case in~\cite{AusTch03} that does not give the dependencies of the constant) that there exists a constant $C = C(N, \norm{A}_{\infty}, \alpha_0, M)$, where $M$ is the Lipschitz constant of $\Omega$, such that
			\[
				\norm{\nabla f}_p \le C \normalnorm{L_2^{1/2} f}_p \qquad \text{for all } f \in D(L_2^{1/2}).
			\]
		Note that in the $L^2$ case the above inequality follows from the reverse inequality by a duality argument. Now, let us consider the case $p \in (1,2)$. It follows from a close and tedious inspection of the proof of~\cite[Theorem~2]{DuoMcI99} that the weak type (1,1) estimate obtained there for $\nabla L_2^{-1/2}$ only depends on $N$ and on upper bounds for $\norm{A}_{\infty}$ and $\alpha_0$ provided the operators satisfy uniform heat kernel estimates (a detailed exposition of this proof can be found in~\cite{Ouh05}). Note that we already know this fact from the proof of Theorem~\ref{thm:r_boundedness_dirichlet}. Hence, for $p \in (1, 2)$, it follows from interpolation that
			\[
				\normalnorm{\nabla f}_p \le C \normalnorm{L_2^{1/2} f}_p \qquad \text{for all } f \in D(L_2^{1/2})
			\]
		for some constant $C = C(p, N, \norm{A}_{\infty}, \alpha_0, M)$. Since $D(L_2^{1/2})$ is a core for $L_p^{1/2}$, the above identity gives by density
			\[
				\norm{\nabla f}_p \le C \normalnorm{L_p^{1/2} f}_p \qquad \text{for all } f \in D(L_p^{1/2}). \qedhere
			\]
	\end{proof}
	
	\begin{remark}
		One even obtains equivalence between $D(L_p^{1/2})$ and $W_0^{1,p}(\Omega)$ for an interval of the form $p \in (1, 2 + \epsilon)$, where $\epsilon$ decreases with the dimension. However, for $N \ge 2$ the result cannot be extended to the full range $p \in (1, \infty)$ as Example~\ref{ex:counterexample_regularity} shows. Further, the Kato square root property holds for complex coefficients in the range $p \in (1, 2 + \epsilon)$ provided the leading coefficients are assumed to lie in $VMO$. For further recent results, including mixed boundary conditions on rough domains, see~\cite{ANHR15}.
	\end{remark}
	
	We are now ready to verify the Hölder estimate. Theorem~\ref{thm:abstract_mr_non-zero} then gives maximal regularity on $L^p(\Omega)$ for $p \ge 2$, which is the crucial range for applications. Note that we only need the estimate given in Theorem~\ref{thm:kato_square_root} in the range $p \in (1,2]$.
	
	\begin{theorem}\label{thm:mr_elliptic}
		For $N \ge 2$ let $\Omega \subset \IR^N$ be a bounded domain with $C^1$-boundary and $p \in [2, \infty)$. Further let $A\colon [0,T] \times \Omega \to \IR^{N \times N}$ be bounded and measurable with the following properties.
		\begin{thm_enum}
			\item The ellipticity estimate~\eqref{eq:ellipticity_condition} for $A(t, \cdot)$ holds uniformly in $t \in [0,T]$.
			\item The vmo-moduli of the components of $A(t,\cdot)$ can be dominated on $(0,1]$ by a function $\eta\colon (0,1] \to \overline{\IR_{\ge 0}}$ with $\lim_{r \downarrow 0} \eta(r) = 0$.
			\item There exist constants $C \ge 0$ and $\beta \in (1/2,1]$ such that for all $s,t \in [0,T]$
				\[
					\norm{A(t,\cdot) - A(s,\cdot)}_{\infty} \le C \abs{t-s}^{\beta}.
				\]
		\end{thm_enum}
		Then the non-autonomous second order problem in divergence form
    		\begin{equation*}
    			\left\{
    			\begin{aligned}
    				u_t(t,x) - \divv(A(t,x) \nabla u(t,x)) & = f(t,x) & \text{on } (0,T) \times \Omega \\
					u(t,x) & = 0 & \text{on } (0,T) \times \partial \Omega \\
    				u(0,x) & = u_0(x) & \text{on } \Omega
    			\end{aligned}
    			\right.
    		\end{equation*}
		has maximal $L^q(L^p)$-regularity for $q \in (1, \infty)$, i.e.\ for all $f \in L^q([0,T];L^p(\Omega))$ and all $u_0 \in D(A(0),X)_{1/q,q}$ the problem has a unique solution $u \in W^{1,q}([0,T];L^p(\Omega))$ with $u(t) \in D(L_p(t))$ for almost all $t \in [0,T]$ and $L_p(\cdot)u(\cdot) \in L^q([0,T];L^p(\Omega))$. Furthermore, there exists a constant $C_{p,q} > 0$ independent of $f$ and $u_0$ with
			\begin{align*}
				\MoveEqLeft \norm{u}_{W^{1,q}([0,T];L^p(\Omega))} + \norm{L_p(\cdot)u(\cdot)}_{L^q([0,T];L^p(\Omega))} \\
				& \le C_{p,q} (\norm{u_0}_{D(A(0),X)_{1/q,q}} + \norm{f}_{L^q([0,T];L^q(\Omega))}).
			\end{align*}
	\end{theorem}
	\begin{proof}
    	Let $p \in [2,\infty)$ and $q$ be its Hölder conjugate given by $\frac{1}{p} + \frac{1}{q} = 1$. For the moment fix $t \in [0,T]$ and let $K_q(t) = L_p(t)'$, which is the realization of $-\divv(A^T(t,\cdot) \nabla \cdot)$. We notice that $A^T\colon [0,T] \times \Omega \to \IR^{N \times N}$ satisfies the same assumptions as $A$ does. Now pick some $g \in D(K_2(t)) \subset D(K_q(t))$. Writing $A(t)$ for the function $A(t, \cdot)$, it follows from the definition of $K_2(t)$ that
    	\begin{equation*}
    		\int_{\Omega} K_q(t) g \phi = \int_{\Omega} A(t)^T \nabla g \cdot \nabla \phi \qquad \text{for all } \phi \in C_c^{\infty}(\Omega).
    	\end{equation*}
    	The space $D(L_2(t))$ is a core for $D(L_q(t))$. Hence, for every $g \in D(L_q(t))$ there exists a sequence $(g_n)_{n \in \IN} \subset D(L_2(t))$ with $g_n \to g$ in $D(L_q(t))$. Now, as a consequence of Theorem~\ref{thm:kato_square_root} one has $g_n \in W_0^{1,q}(\Omega)$ and $g_n \to g \in W_0^{1,q}(\Omega)$. Taking limits, we obtain for all $\phi \in C_c^{\infty}(\Omega)$
    	\begin{equation*}
    		\int_{\Omega} K_q(t) g \phi = \lim_{n \to \infty} \int_{\Omega} K_q(t) g_n \phi = \lim_{n \to \infty} \int_{\Omega} A(t)^T \nabla g_n \cdot \phi = \int_{\Omega} A(t)^T \nabla g \cdot \nabla \phi.
    	\end{equation*}
    	By density, this identity extends from all test functions $\phi \in C_c^{\infty}(\Omega)$ to all functions $f \in W_0^{1,p}(\Omega)$. For $f \in D(L_p(s))$ and $g \in D(K_q(t))$ the inclusion $D(L_p(s)) \hookrightarrow W_0^{1,p}(\Omega)$ given by Proposition~\ref{prop:sobolev_elliptic_operator} allows us to write
    	\begin{align*}
    		\langle (L_p)_{-1}(t)f, g \rangle = \int_{\Omega} f K_q(t) g = \int_{\Omega} A(t)^T \nabla g \cdot \nabla f = \int_{\Omega} A(t) \nabla f \cdot \nabla g.
    	\end{align*}
    	Further, as a consequence of the same inclusions and by the same reasoning we have
    	\begin{align*}
    		\langle (L_p)_{-1}(s)f, g \rangle = \langle L_p(s)f, g \rangle = \int_{\Omega} A(s) \nabla f \cdot \nabla g.
    	\end{align*}
    	Altogether we therefore obtain the estimate
    	\begin{align*}
    		\MoveEqLeft \abs{\langle L_{-1}(t) - L_{-1}(s)f,g \rangle} = \abs{\int_{\Omega} (A(t) - A(s)) \nabla f \cdot \nabla g} \\
    		& \le \norm{A(t) - A(s)}_{\infty} \norm{\nabla f}_{p} \norm{\nabla g}_{q} \lesssim \norm{A(t) - A(s)}_{\infty} \norm{f}_{D(L_p(s))} \norm{g}_{D({L_p(t)'}^{1/2})} \\
			& \lesssim \abs{t-s}^{\beta } \norm{f}_{D(L_p(s))} \norm{g}_{D({L_p(t)'}^{1/2})}.
    	\end{align*}
		Now, the assertion follows from Theorem~\ref{thm:abstract_mr_non-zero} applied to the complex interpolation scale, Proposition~\ref{prop:bip_fractional} and Theorem~\ref{thm:bip_second_order}.
	\end{proof}
	
	\begin{remark}\label{rem:elliptic_less_regularity}
		In the case $p = 2$ Proposition~\ref{prop:sobolev_elliptic_operator} holds for measurable coefficients as a consequence of the fact that the domain of the associated operator embeds into the form domain. Hence, in the Hilbert space case we essentially get the same results as those that can be obtained with the results in~\cite{HaaOuh15}. Further it follows from Remark~\ref{rem:more_general_estimate} that the VMO assumption on the coefficients can be loosened for one spatial component.
	\end{remark}
	
\subsection{Comments on the one dimensional case}	

	It is known that for $N = 1$ Theorem~\ref{thm:kato_square_root} and even the Kato square root property holds for all $p \in (1, \infty)$~\cite[Proposition~5.17]{Aus07}. Hence, for $N = 1$ the result of Theorem~\ref{thm:mr_elliptic} can be obtained for measurable coefficients and all $p \in (1, \infty)$ assuming the same dependence on the constants. However, we do not know any explicit reference for this. Even more, one would get that both $L_p(t)$ and $L_p(t)'$ satisfy the Acquistapace--Terreni condition. Hence, as in Remark~\ref{rem:weights}, for $N = 1$ one would obtain maximal regularity for certain weights. Compare this with Remark~\ref{rem:elliptic_less_regularity} where also one spatial component is allowed to have less regularity.
				
	\emergencystretch=0.75em
	\printbibliography

\end{document}